\documentclass[11pt,reqno]{amsart}
\usepackage{amsmath, amsfonts, amsthm, amssymb,amscd, graphicx, amscd}
\usepackage{float,epsf}
\usepackage[english]{babel}
\usepackage{enumerate}
\usepackage{tikz}
\usepackage{mathrsfs}
\usepackage[numbers,sort&compress]{natbib}

\usepackage{srcltx}
\usepackage{geometry}
\usepackage{verbatim}
\usepackage{mathrsfs}
\usepackage{hyperref}
\usepackage{enumitem} 
\usepackage{bbm} 
\usepackage{stmaryrd}

\usepackage{relsize}
\usepackage{exscale}

\usepackage{mathtools}

\newtheorem{thm}{Theorem}[section]
\newtheorem{defn}[thm]{Definition}

\newtheorem{lem}[thm]{Lemma}

\newtheorem{rem}[thm]{Remark}

\numberwithin{equation}{section}

\newcommand{\R}{\mathbb{R}}

\def\dd{{\rm d}}
\hypersetup{bookmarksdepth=2}

\begin{document}
\title{Singular limit for nonlocal conservation laws with nonlocality in density and velocity}
\author{Immanuel Ben-Porat}
\address{Immanuel Ben-Porat, The University of Texas at Austin, 2515 Speedway, Austin, TX 78712.}
\email{immanuel.benporat@austin.utexas.edu}
\maketitle
\begin{abstract}
We study the singular limit of a generalized nonlocal conservation law introduced in  \cite{friedrich2022conservation}. The nonlocal equation incorporates non-locality both in density and velocity. We derive entropy solutions of the underlying Burgers equation as a singular limit. The main ingredient of the proof is a nonlocal Oleinik type estimate, extending the method of  \cite{coclite2024oleinik}. The nonlocal Oleinik estimate  provides uniform bounds on the total variation, thereby enabling to pass to the limit in the regularization parameter.      
\end{abstract}
\section{Introduction}
\subsection{Main objective.} In this work we study the singular limit of the following scalar nonlocal conservation law
\begin{align}
\begin{cases}
\partial_{t}u_{\varepsilon}(t,x)+\partial_{x}(V_{1}(V_{2}(u_{\varepsilon})\ast \eta_{\varepsilon})u_{\varepsilon}(t,x))=0, \ (t,x)\in (0,\infty)\times \mathbb{R} \\ u_{\varepsilon}(0,x)=u^{0}, \ x\in \mathbb{R}.     
\end{cases}
\label{nonlocal Cauhy problem}  
\end{align}
Here $V_{1},V_{2}:\mathbb{R}\rightarrow \mathbb{R}$ are  given \textbf{velocity functions}, $\eta\in L^{1}(\mathbb{R})$ is an integrable function with $\int_{\mathbb{R}}\eta(x)\ \dd x=1$ and 
\begin{align}
\eta_{\varepsilon}(x)\coloneqq \frac{1}{\varepsilon}\eta(\frac{x}{\varepsilon}).
\label{def of etaeps}    
\end{align}
The local counterpart of \eqref{nonlocal Cauhy problem} is the nonlinear conservation law 
\begin{align}
\begin{cases}
\partial_{t}u(t,x)+\partial_{x}(V_{1}\circ V_{2}(u(t,x))u(t,x))=0, \ (t,x)\in (0,\infty)\times \mathbb{R}\\
u(0,x)=u^{0}, \ x\in \mathbb{R}. 
\end{cases}  
\label{local Cauchy}
\end{align}
Since $\eta_{\varepsilon}(x)\coloneqq \frac{1}{\varepsilon}\eta(\frac{x}{\varepsilon})$ is an approximation of the identity, it is noticeable that \eqref{local Cauchy} is formally derived from \eqref{nonlocal Cauhy problem} in the limit as $\varepsilon\rightarrow 0$. The rigorous justification of this limit is known as the \textbf{singular limit} problem.  Equation \eqref{nonlocal Cauhy problem} was introduced in \cite{friedrich2022conservation} as a more general model for traffic flow. In the same work the authors build an appropriate well-posedness theory (see Theorem \ref{Reminder of theorem}) and ask to justify the singular limit leading from \eqref{nonlocal Cauhy problem} to \eqref{local Cauchy}. The main aim of this work is to give an affirmative answer to this question by proving the singular limit for entropy solutions. To understand the motivation behind \eqref{nonlocal Cauhy problem} from a modeling point of view, it is instructive to consider the following two  special cases of \eqref{nonlocal Cauhy problem}. The first case is where $V_{2}=\mathrm{Id}$, so that \eqref{nonlocal Cauhy problem} reduces to 
\begin{align}
\begin{cases}
\partial_{t}u_{\varepsilon}(t,x)+\partial_{x}(V_{1}(u_{\varepsilon}\ast \eta_{\varepsilon})u_{\varepsilon}(t,x))=0, \ (t,x)\in (0,\infty)\times \mathbb{R} \\ u_{\varepsilon}(0,x)=u^{0}, \ x\in \mathbb{R}.     
\end{cases}
\label{nonlocal Cauchy V2 identity}  
\end{align}
The second case is where $V_{1}=\mathrm{Id}$,  so that \eqref{nonlocal Cauhy problem} reduces to 
\begin{align}
\begin{cases}
\partial_{t}u_{\varepsilon}(t,x)+\partial_{x}((V_{2}(u_{\varepsilon})\ast \eta_{\varepsilon})u_{\varepsilon}(t,x))=0, \ (t,x)\in (0,\infty)\times \mathbb{R} \\ u_{\varepsilon}(0,x)=u^{0}, \ x\in \mathbb{R}.     
\end{cases}
\label{nonlocal Cauchy V1 identity}  
\end{align}
We interpret $u_{\varepsilon}$ as the density of vehicles, $u_{\varepsilon}\ast \eta_{\varepsilon}$ as the averaged density and $V_{2}(u_{\varepsilon})\ast \eta_{\varepsilon}$ as an averaged velocity. Under this interpretation, equation \eqref{nonlocal Cauchy V2 identity} says that drivers react to the averaged density whereas equation \eqref{nonlocal Cauchy V1 identity} says that drivers react to the averaged velocity. Thus, the model \eqref{nonlocal Cauhy problem} is more robust, as it allows drivers to react both to the averaged density \textbf{and} averaged velocity.  
\subsection{Background}
Prior to presenting our main result we give a brief general overview of the singular limit problem. Let us first consider the case where $V_{2}=\mathrm{Id}$, which is the case most commonly considered in the literature. 
The local and global well-posedness theory for equations of the type \eqref{nonlocal Cauchy V2 identity} is well established,  including in multidimensional settings and under minimal regularity and structural assumptions on $\eta$ - see \cite{blandin2016well, crippa2013existence, coclite2022existence,colombo2024multidimensional,keimer2017existence,keimer2018existence, coclite2025singular, goatin2016well}. In addition, Kru{\v{z}}kov's classical theory of entropy solutions \cite{kruvzkov1970first} identifies a class of weak solutions for which \eqref{local Cauchy} enjoys global well-posedness.  
\\  
If $u^{0}$ has finite mass, that $u_{\varepsilon}$ admits a weak-$\ast$ limit in $L^{\infty}([0,T];\mathcal{M}(\mathbb{R}))$ (up to a subsequence) follows directly from conservation of the $L^{1}$ norm and the Banach-Alaoglu theorem. Within itself, this is not enough in order to pass to the limit as $\varepsilon\rightarrow 0$ in the nonlinear term $V_{1}(u_{\varepsilon}\ast \eta_{\varepsilon})u_{\varepsilon}$, because weak convergence is ill behaved under nonlinear operations. A first challenge that one is faced with is to show that this limit is indeed a solution to \eqref{local Cauchy} -- this task is often nontrivial even when dealing with classical solutions. A subsequent challenge consists in showing that this derivation is also valid for entropy admissible solutions of \eqref{local Cauchy} -- in this case one seeks a global in time limit. Identifying  the precise assumptions which need to be imposed on $\eta,V_{1}$ and $u^{0}$ for this convergence to be valid is a main challenge in the singular limit problem. 
Among several important works in recent years, which include both positive results as well as counterexamples to the singular limit we mention \cite{colombo2019singular, colombo2021local,coclite2024oleinik, friedrich2024conservation, coclite2025singular, bressan2020traffic, BressanShen2021, keimer2019approximation}. See also \cite{colombo2023overview} for a general overview of the state of the art of the problem. Of particular importance to the present work is \cite{coclite2024oleinik}, in which the authors prove an \textbf{Oleinik type estimate} for the nonlocal equation \eqref{nonlocal Cauchy V2 identity} by analogy with the classical Oleinik inequalities for the Burgers equation \cite{Oleinik1957}. In more detail, the authors consider the exponential kernel $\eta(x)=\mathbf{1}_{(-\infty,0]}(x)e^{x}$ and the nonlocal density $w_{\varepsilon}\coloneqq u_{\varepsilon}\ast \eta_{\varepsilon}$. Then, subject to concavity assumptions on the flux function $\xi\mapsto \xi V_{1}(\xi)$ (which, interestingly, are more restrictive than in the local settings) it is proved that $w_{\varepsilon}$ satisfies the following one sided Lipschitz estimate uniformly in $\varepsilon$:
\begin{align}
\frac{w_{\varepsilon}(t,x)-w_{\varepsilon}(t,y)}{x-y}\geq -\frac{C}{t}\ \mbox{for all}\ x>y\ \mbox{and}\ t\in (0,\infty). \label{nonlocal Oleinik ine} \end{align}
This one sided inequality can be used to gain uniform in $\varepsilon$ bounds on the total variation of $w_\varepsilon$, which in turn allows to pass to the limit as $\varepsilon\rightarrow 0$. It is crucial to point out that the nonlocal Oleinik inequality \eqref{nonlocal Oleinik ine} is valid for $L^{\infty}$ initial data, thereby allowing to prove the singular limit for a class of initial data with very mild regularity assumptions.  Nonlocal Oleinik inequalities for the \textbf{local} density $u_{\varepsilon}$ have been proved in \cite{colombo2021local} but under a one sided Lipschitz condition on the initial data. See also \cite{coclite2023long} for nonlocal Oleinik estimates in the context of long time asymptotics.   
\vspace{0.3 cm}

The model \eqref{nonlocal Cauchy V1 identity} appeared later in \cite{friedrich2024conservation}. In the same work the authors prove the singular limit by two different methods: first they consider the case of the exponential kernel and general $\mathrm{BV}$ initial data; then they consider monotone initial data and general anisotropic kernels. In the former case they prove uniform in $\varepsilon$ bounds for the total variation of the nonlocal density, whereas in the latter case they prove that monotonicity is preserved in time which in turn provides information about the total variation of the local density $u_{\varepsilon}$. Nonlocality in velocity is also a feature of the Aw-Rascle model introduced in \cite{chiarello2020micro}. The relevant rigorous singular limit analysis of this model was carried out in \cite{amadori2026systems}.  
As already mentioned, the more general model, unifying \eqref{nonlocal Cauchy V2 identity}-\eqref{nonlocal Cauchy V1 identity} has been introduced in \cite{friedrich2022conservation}. In the same work,  and more specifically in Remark 2.5, the authors observe that \eqref{nonlocal Cauhy problem} is monotonicity preserving. This can be used in order to gain uniform in $\varepsilon$ bounds on the total variation of  $u_{\varepsilon}$, which would then yield the singular limit via a compactness argument in a manner similar to what has been described earlier. 
\subsection{Main Results}
In this work we prove the singular limit leading from \eqref{nonlocal Cauhy problem} to \eqref{local Cauchy} via a nonlocal Oleinik inequality approach. In what follows we define the following auxiliary function $h:I\rightarrow \mathbb{R}$: 
\begin{align*}
h(\xi)\coloneqq-V_{2}^{-1}(\xi)V_{2}'(V_{2}^{-1}(\xi))\  \mbox{and}\ I=\left[V_{2}(\mathrm{ess}\sup u^{0}),V_{2}(0)\right].    
\end{align*}
\begin{thm}\label{oleinik ine}
Suppose there are constants $\alpha>0,\beta>0,\gamma>0$ such that: 
\begin{itemize}
    \item[1.] $u^{0}\in L^{\infty}(\mathbb{R})$, $u^{0}\geq 0$.
    \item[2.] $\eta$ is the exponential kernel, i.e $\eta(x)=\mathbf{1}_{(-\infty,0]}(x)e^{x}$.
    \item[3.] $V_{1},V_{2}\in W^{3,\infty}_{\mathrm{loc}}(\mathbb{R})$. 
   \item[4.] $V_{2}'(\xi)\leq -\beta$ for all $\xi\in [0,\mathrm{ess}\sup u^{0}]$.  
    \item[5.] $h'(\xi)\leq -\gamma$ for all $\xi \in I$. Furthermore, $h''(\xi)\leq 0$ for a.e. $\xi\in I$. 
    \item[6.] $V_{1}'(\xi)\geq \alpha$, $V_{1}''(\xi)\geq 0$ for all $\xi \in I$ and $0\leq D\coloneqq \underset{\xi\in I}{\sup}V_{1}''(\xi)h(\xi)<\alpha \gamma$. 
 
\end{itemize}
Let $u_{\varepsilon}$ be the unique weak solution to \eqref{nonlocal Cauhy problem} with initial data $u^{0}$ (as guaranteed by Theorem \ref{Well posedness of nonlocal eq}) and let $W_{\varepsilon}\coloneqq  V_{2}(u_{\varepsilon})\ast \eta_{\varepsilon}$. Let $u$ be the unique entropy solution of \eqref{local Cauchy} with initial data $u^{0}$. Then, 
\begin{enumerate}
    \item It holds that 
    \begin{align}
    \frac{W_{\varepsilon}(t,x)-W_{\varepsilon}(t,y)}{x-y}\leq \frac{1}{\kappa t}\ \mbox{for all}\ x>y\ \mbox{and all}\ t\in (0,\infty) \label{nonlocal oleinik W statement}  
\end{align}
where 
\begin{align}
    \kappa \coloneqq \alpha \gamma -D. \label{def of kappa}    
\end{align}
    \item It holds that  $u_{\varepsilon}\underset{\varepsilon \rightarrow 0}{\rightarrow} u$ in $L^{1}_{\mathrm{loc}}((0,\infty)\times \mathbb{R})$.   
\end{enumerate}
\end{thm}
A few remarks are in order regarding Theorem \ref{oleinik ine}:
\begin{enumerate}
    \item[1.] The proof of the singular limit is already not entirely trivial in the simplified settings of classical solutions, i.e. when assuming that the initial data enjoys at least some Lipschitz regularity and proving the convergence of $u_{\varepsilon}$ to $u$ on any time interval before formation of shock. Indeed, a typical argument for the singular limit for classical solutions would be based on obtaining a Gr\"onwall estimate of the form $\frac{\dd}{\dd t}\left\Vert (u_{\varepsilon}-u)(t,\cdot)\right\Vert_{\infty}\lesssim \left\Vert (u_{\varepsilon}-u)(t,\cdot)\right\Vert_{\infty}+o_{\varepsilon}(1)$, where the remainder $o_{\varepsilon}(1)$ is negligible precisely due to the fact that $u$ is regular enough. An estimate of this type does not close for \eqref{nonlocal Cauhy problem}, and a key idea to overcome this difficulty rests on studying the equation governing $V_{2}(u_{\varepsilon})-V_{2}(u)$ and obtaining a Gr\"onwall estimate on $\left\Vert V_{2}(u_{\varepsilon}(t,\cdot))-V_{2}(u(t,\cdot))\right\Vert_{\infty}$. 
 \item[2.]  Theorem \ref{oleinik ine} is already new in the simpler case where $V_{1}=\mathrm{Id}$. Indeed, we stress that the method of \cite{friedrich2024conservation} is valid only for $\mathrm{BV}$ initial data, in which case one can show that the total variation of $W_{\varepsilon}(t,\cdot)$ is decreasing in time. By comparison, the nonlocal Oleinik inequality approach enables us to handle merely $L^{\infty}$ initial data. \item[3.] The strict concavity of the flux function $\xi \mapsto \xi V_{1}\circ V_{2}(\xi)$ is a byproduct of the assumptions imposed in Theorem \ref{oleinik ine} and is therefore not stated as a separate assumption. See Lemma \ref{Concavity of flux function} for details about this implication. The concavity of the flux function is essential in order to make usage of Oleinik inequalities. 
 \item[4.] It is not difficult to find $V_{1},V_{2}$ realizing the assumptions in Theorem \ref{oleinik ine} -- see Remark \ref{example of V1 V2}. 
\end{enumerate}
In view of point 1. we also include a proof of the singular limit for classical solutions. This is both for pedagogical reasons and concrete reasons: pedagogical, because the classical settings already demonstrate some curious genuine features of \eqref{nonlocal Cauhy problem}; concrete, because in the classical settings we are able to treat general anisotropic kernels (see \eqref{Isentropy condition} below for definition), and not only the exponential kernel. In this context, we remark that proving nonlocal Oleinik inequalities for more general anisotropic kernels than the exponential kernel remains an open problem.    
\begin{thm}\label{MAIN THM}
Suppose that: 
\begin{itemize}
    \item[i.] $u^{0}\in W^{2,\infty}(\mathbb{R})$, $\partial_{x}u^{0}\in \mathrm{BV}(\mathbb{R})$ and $u^{0}\geq 0$.
    \item[ii.] $\eta$ satisfies \eqref{Isentropy condition}. In addition, $\eta\in \mathrm{Lip}(\mathbb{R}_{-})$ and $\eta$ has finite first moments, i.e. $\int_{\mathbb{R}}\left\vert x\right\vert\eta(x)\ \dd x<\infty $.  
    \item[iii.] $V_{1},V_{2}\in W^{3,\infty}_{\mathrm{loc}}(\mathbb{R})$ and $V_{1}'(\xi)\geq 0, V_{2}'(\xi)\leq 0$ for all $\xi \in [\lambda_{\min},\lambda_{\max}]$ where we have set 
    \begin{align*}
    \lambda_{\min}\coloneqq \min\left\{\inf u^{0},\min_{r\in [\inf u^{0},\sup u^{0}]}V_{2}(r)\right\},\ \lambda_{\max}\coloneqq \max\left\{\sup u^{0},\max_{r\in [\inf u^{0},\sup u^{0}]}V_{2}(r)\right\}.    
\end{align*}
\item[iv.] $V_{2}$ is biLipschitz on $[0,\sup u^{0}]$, i.e. there is some $L>0$ such that 
\begin{align}
\frac{1}{L}\left\vert \xi-\xi'\right\vert\leq \left\vert V_{2}(\xi)-V_{2}(\xi')\right\vert\leq L\left\vert \xi-\xi'\right\vert\ \mbox{for all}\ (\xi,\xi')\in [0,\sup u^{0}]\times [0,\sup u^{0}].   \label{bilip condition on V2}        
\end{align}

\end{itemize}
Let $u_{\varepsilon}$ be a solution to \eqref{nonlocal Cauhy problem} with initial data $u^{0}$. Let $u$ be the unique solution of \eqref{local Cauchy} with initial data $u^{0}$ and let $T_{\ast}>0$ be the maximal time on which $u$ is uniformly $W^{2,\infty}$. Then, for any $T<T_{\ast}$ it holds that 
\begin{align*}
\underset{t\in [0,T]}{\sup} \left\Vert (u_{\varepsilon}-u)(t,\cdot)\right\Vert_{\infty}=O(\varepsilon).     
\end{align*}
\end{thm}
 \section{Preliminaries}
 We recall that the kernel $\eta$ will be called \textbf{anisotropic} if it satisfies the following requirements:  
\begin{align}\tag{I}
 \eta\in L^{1}(\mathbb{R})\cap L^{\infty}(\mathbb{R}),\ \mathrm{supp}(\eta)\subset \mathbb{R}_{-}, \int_{\mathbb{R}}\eta(x)\ \dd x=1 \ \mbox{and}\ \eta\geq0 \ \mbox{is non-decreasing on} \ \mathbb{R}_{-}.   
\label{Isentropy condition}
\end{align}
The global well-posedness theory of \eqref{nonlocal Cauhy problem} has been investigated in \cite{friedrich2022conservation}. The anisotropy of $\eta$ already plays a role in the well-posedness theory, as it provides a maximum principle.
\begin{thm} \textup{(\cite{friedrich2022conservation}, Theorem 2.3)}
Suppose that: 
\begin{itemize}
\item[i.] $u^{0}\in L^{\infty}(\mathbb{R})\cap \mathrm{BV}(\mathbb{R})$ and $u^{0}\geq 0$. 
\item[ii.] $\eta$ satisfies \eqref{Isentropy condition}.  
\item[iii.] $V_{1},V_{2}\in W^{3,\infty}_{\mathrm{loc}}(\mathbb{R})$ are such that $V_{1}'(\xi)\geq 0$ for all $\xi\in I$ and $V_{2}'(\xi)\leq 0$ for all $\xi \in [0,\mathrm{ess}\sup u^{0}]$. 
\end{itemize}
Then, there exist a unique weak solution $u_{\varepsilon}\in C([0,T];L^{1}(\mathbb{R}))\cap L^{\infty}([0,T];\mathrm{BV}(\mathbb{R}))$. Moreover, it holds that 
\begin{align*}
\mathrm{ess}\inf u^{0}\leq u_{\varepsilon}(t,x)\leq \mathrm{ess}\sup u^{0}\ \mbox{for all}\ \varepsilon>0 \ \mbox{and a.e.}\ (t,x)\in [0,T]\times \mathbb{R}. 
\end{align*}
If in addition $u^{0}\in W^{2,\infty}(\mathbb{R})$ then $u_{\varepsilon}\in L^{\infty}_{\mathrm{loc}}([0,\infty);W^{2,\infty}(\mathbb{R}))$. 
\label{Reminder of theorem}
\end{thm}
 The next global well-posedness result is extension of Theorem \ref{Reminder of theorem} to weak solutions with only $L^{\infty}$ initial data. It is followed by an approximation lemma which will allow us to reduce some of the calculations to initial data satisfying further regularity (see Theorem 4.17 in \cite{keimer2019approximation} for similar results). In both of the following theorems, $\varepsilon>0$ is fixed and plays no essential role. As to make the work streamlined the proofs are postponed to the appendix.  
\begin{thm} \textup{(Weak solutions of \eqref{nonlocal Cauhy problem})}
Suppose that: 
\begin{itemize}
\item[i.] $u^{0}\in L^{\infty}(\mathbb{R})$ and $u^{0}\geq 0$. 
\item[ii.] $\eta(x)=\mathbf{1}_{(-\infty,0]}(x)e^{x}$ .  
\item[iii.] $V_{1},V_{2}\in W^{3,\infty}_{\mathrm{loc}}(\mathbb{R})$ are such that $V_{1}'(\xi)\geq 0$ for all $\xi\in I$ and $V_{2}'(\xi)\leq 0$ for all $\xi \in [0,\mathrm{ess}\sup u^{0}]$. 
\end{itemize}
Then, for any $T>0$ there exists a unique weak solution $u_{\varepsilon}\in L^{\infty}([0,T]\times \mathbb{R})$ with initial data $u^{0}$ to \eqref{nonlocal Cauhy problem}. Moreover, it holds that 
\begin{align*}
\mathrm{ess}\inf u^{0}\leq u_{\varepsilon}(t,x)\leq \mathrm{ess}\sup u^{0}\ \mbox{for all}\ \varepsilon>0 \ \mbox{and a.e.}\ (t,x)\in [0,T]\times \mathbb{R}. 
\end{align*}
\label{Well posedness of nonlocal eq}
\end{thm}
\begin{thm}\textup{(Stability with respect to initial data)}
Let the assumptions of Theorem \ref{Well posedness of nonlocal eq} hold. Let $u_{\varepsilon n}$ and $u_{\varepsilon}$ be bounded weak solutions of \eqref{nonlocal Cauhy problem}, with initial
data $u_{n}^{0}$ and $u^{0}$, respectively. Suppose that $\{u_{n}^{0}\}_{n\geq 1}$ is such that: 
\begin{enumerate}
    \item $0\leq u_{n}^{0}\leq \mathrm{ess}\sup u^{0}\ \mbox{for all}\ n\geq 1. $
    \item $ u_n^{0}\underset{n\rightarrow \infty}{\rightarrow} u^{0}\quad\text{in }L^1_{\mathrm{loc}}(\R)$. 
\end{enumerate}
Then, for every $R>0$ and any $T>0$ it holds that 
\begin{equation}\label{eq:local-conclusion}
 \sup_{0\leq t\leq T}
 \|(u_{\varepsilon n}-u_{\varepsilon})(t,\cdot)\|_{L^1([-R,R])}\underset{n\rightarrow \infty}{\rightarrow}0.
\end{equation}
\label{approximation lemma}
\end{thm}
\begin{rem}
Theorem \ref{Well posedness of nonlocal eq} is not as innocent looking as may appear at first glance: Theorem 2.3 in \cite{friedrich2022conservation} proves global well-posedness for bounded $\mathrm{BV}$ initial data, and it may be tempting to think that a standard approximation argument yields the existence of weak $L^{\infty}$ solutions. However, the situation is more involved: the weak convergence of an approximating sequence $u_{n}$ to a candidate solution $u$ does not automatically imply weak convergence of $V_{2}(u_{n})$ to $V_{2}(u)$. As a result, a more subtle argument is needed to justify the convergence. This is a genuine structural obstruction of the specific equation \eqref{nonlocal Cauhy problem}, which does not arise when $V_{2}$ is affine. The proof of uniqueness also requires further adaptions. Since we wish to concentrate on the singular limit we detain the full argument to the appendix.     
\end{rem}
We finish by recalling the concept of entropy solutions and their equivalent characterization. 
\begin{defn} \label{def of entropy solutions}\textup{(Entropy solutions)}
Suppose that: 
\begin{itemize}
    \item $u^{0}\in L^{\infty}(\mathbb{R})$ and $u^{0}\geq 0$. 
    \item $A\in W^{1,\infty}_{\mathrm{loc}}(\mathbb{R})$.  
\end{itemize} 
We say that $u\in L^{\infty}((0,\infty)\times \mathbb{R})$ is an \textbf{entropy solution} of the Cauchy problem 
\begin{align}
\partial_{t}u+\partial_{x}(A(u))=0,\ u(0,\cdot)=u^{0} \label{Cauchy problem generic}
\end{align}
provided the following holds: 
\begin{enumerate}
    \item For all non-negative $\varphi\in C^{\infty}_{0}(
    (-a,\infty)\times \mathbb{R})$ and all $k\in \mathbb{R}$ it holds that 
    \begin{align*}
     &\int_{0}^{\infty}\int_{\mathbb{R}}\left\vert u-k\right\vert \partial_{\tau}\varphi(\tau,x)\ \dd x\dd \tau\\
     &+\int_{0}^{\infty}\int_{\mathbb{R}}\mathrm{sgn}(u-k)(A(u)-A(k))\partial_{x}\varphi(\tau,x)\ \dd x\dd \tau
     +\int_{\mathbb{R}}\left\vert u^{0}-k\right\vert\varphi(0,x)\ \dd x\geq 0.    
    \end{align*}
    \item For all $\Omega\Subset \mathbb{R}$ it holds that $\underset{t\downarrow 0}{\lim}\left\Vert u(t,\cdot)-u^{0}\right\Vert_{L^{1}(\Omega)}=0$. 
\end{enumerate}
\end{defn}
\begin{thm} \textup{(\cite{evans2022partial},§3.4.3, Theorem 3)} Let the conditions of Definition \ref{def of entropy solutions} hold and assume in addition that 
\begin{itemize}
    \item[i.] $A\in C^{2}(\mathbb{R})$. 
    \item[ii.] There exists some $\delta>0$ such that $A''(\xi)\leq -\delta $ for all $\xi\in [0,\mathrm{ess}\sup u^{0}]$.
\end{itemize}
Then, $u$ is an entropy solution of \eqref{Cauchy problem generic} iff $u$ is a weak solution of \eqref{Cauchy problem generic} and for a.e. $x>y$ and a.e. $t>0$ it holds that 
\begin{align*}
\frac{u(t,x)-u(t,y)}{x-y}\geq -\frac{1}{\delta t}.     
\end{align*}
\label{equiv chara}
\end{thm}
\begin{rem}
The landmark work of Kru{\v{z}}kov \cite{kruvzkov1970first} yields the global existence and uniqueness of entropy solutions.
\end{rem}
\section{Proof of Theorem \ref{MAIN THM}}
As a warm-up, we present in this section a proof of the singular limit for classical solutions. As already remarked, working with classical solutions simplifies matters significantly. Nevertheless the proof still exhibits some differences in comparison to the singular limit of \eqref{nonlocal Cauchy V2 identity}. In the sequel we invoke the following abbreviation: given a function $f\in L^{\infty}_{\mathrm{loc}}(\mathbb{R})$ the notation $\left\Vert f\right\Vert_{\infty,\mathrm{loc}}$ means $\left\Vert f\right\Vert_{L^{\infty}(K)}$ for some compact interval $K\Subset \mathbb{R}$ which we do not specify for better readability. 
\begin{lem} 
\label{limit of solutions}
Let the assumptions of Theorem \ref{MAIN THM} hold. Let $u_{\varepsilon}$ and $u$ be the classical solutions of \eqref{nonlocal Cauhy problem} and \eqref{local Cauchy} respectively on some time interval $[0,T]$. Then, for any $t\in [0,T]$ it holds that 
\begin{align*}
\underset{x\rightarrow \pm \infty}{\lim}u_{\varepsilon}(t,x)=\underset{x\rightarrow \pm \infty}{\lim}u(t,x)=\underset{x\rightarrow \pm \infty}{\lim} u^{0}(x). 
\end{align*}
\end{lem}
\begin{proof}
It is readily checked that the assumption $\partial_{x}u^{0}\in \mathrm{BV}(\mathbb{R})$ is propagated in time, i.e. that for any $t\in [0,T]$ it holds that $\partial_{x}u_{\varepsilon}(t,\cdot)\in \mathrm{BV}(\mathbb{R})$ and $\partial_{x}u(t,\cdot)\in \mathrm{BV}(\mathbb{R})$.  Integrating in $t$ equation \eqref{nonlocal Cauhy problem} yields  
\begin{align}
u_{\varepsilon}(t,x)&=u^{0}(x)-\int_{0}^{t}\partial_{x}(V_{1}(W_{\varepsilon})(\tau,x)u_{\varepsilon}(\tau,x))\ \dd \tau \notag\\
&=u^{0}(x)-\int_{0}^{t}V_{1}(W_{\varepsilon}(\tau,x))\partial_{x}u_{\varepsilon}(\tau,x)\ \dd \tau-\int_{0}^{t}V_{1}'(W_{\varepsilon}(\tau,x))\partial_{x}W_{\varepsilon}(\tau,x) u_{\varepsilon}(\tau,x)\ \dd \tau. 
\label{uepsilon integrated in t}
\end{align}
Since $\partial_{x}u_{\varepsilon}(\tau,\cdot)\in \mathrm{BV}(\mathbb{R})$ it must be that 
\begin{align}
\underset{x\rightarrow \pm \infty}{\lim}\partial_{x}u_{\varepsilon}(\tau,x)=0.  \label{vanishing of partialxueps} 
\end{align}
Furthermore, since $u_{\varepsilon}(\tau,\cdot)$ has finite total variation, it admits limits at $\pm \infty$ so that we may write 
\begin{align*}
\underset{x\rightarrow  \infty}{\lim} u_{\varepsilon}(\tau,x)=\ell^{+}_{\varepsilon}(\tau)\ \mbox{and}\ \underset{x\rightarrow -\infty}{\lim}u_{\varepsilon}(\tau,x)=\ell_{\varepsilon}^{-}(\tau).    
\end{align*}
Therefore, by the dominated convergence theorem and continuity of $V_{2}$ we get  
\begin{align}
\underset{x\rightarrow \pm \infty}{\lim}\partial_{x}W_{\varepsilon}(\tau,x)&=\underset{x\rightarrow \pm \infty}{\lim}\left[\frac{1}{\varepsilon^{2}}\int_{x}^{\infty}\eta'(\frac{x-y}{\varepsilon})V_{2}(u_{\varepsilon}(\tau,y))\ \dd y-\frac{\eta(0^{-})V_{2}(u_{\varepsilon}(\tau,x))}{\varepsilon}\right] \notag\\
&=\underset{x\rightarrow \pm \infty}{\lim}\left[\frac{1}{\varepsilon^{2}}\int_{-\infty}^{0}V_{2}(u_{\varepsilon}(\tau,x-y))\eta'(\frac{y}{\varepsilon})\ \dd y-\frac{\eta(0^{-})V_{2}(u_{\varepsilon}(\tau,x))}{\varepsilon}\right] \notag\\
&=\frac{\eta(0^{-})V_{2}(\ell_{\varepsilon}^{\pm}(\tau))}{\varepsilon}-\frac{\eta(0^{-})V_{2}(\ell_{\varepsilon}^{\pm}(\tau))}{\varepsilon}=0. \label{vanishing of partialW}
\end{align}
Substituting \eqref{vanishing of partialxueps} and \eqref{vanishing of partialW} inside \eqref{uepsilon integrated in t} we conclude that 
\begin{align}
\underset{x\rightarrow \pm \infty}{\lim}u_{\varepsilon}(t,x)=\underset{x\rightarrow \pm \infty}{\lim}u^{0}(x).  \label{limueps}   
\end{align}
In addition, integrating in $t$ equation \eqref{local Cauchy} yields  
\begin{align*}
u(t,x)&=u^{0}(x)-\int_{0}^{t}\partial_{x}(V_{1}(V_{2}(u(\tau,x)))u(\tau,x))\ \dd \tau\\
&=u^{0}(x)-\int_{0}^{t}V_{1}'(V_{2}(u(\tau,x)))V_{2}'(u(\tau,x))\partial_{x}u(\tau,x) u(\tau,x)\ \dd \tau\\
&-\int_{0}^{t}V_{1}(V_{2}(u(\tau,x)))\partial_{x}u(\tau,x)\ \dd \tau.     
\end{align*}
So, using that $\underset{x\rightarrow \pm \infty}{\lim}\partial_{x}u(\tau,x)=0$ we deduce
\begin{align}
\underset{x\rightarrow \pm \infty}{\lim} u(t,x)=\underset{x\rightarrow \pm \infty}{\lim} u^{0}(x). \label{limu}    
\end{align}
Combining \eqref{limueps} and \eqref{limu} gives the asserted result. 
\end{proof}

\textbf{Proof of Theorem \ref{MAIN THM}}. The basic idea underpinning the proof is to study the equation governing $V_{2}(u_{\varepsilon})-V_{2}(u)$ and obtain convergence of $V_{2}(u_{\varepsilon})$ to $V_{2}(u)$ as $\varepsilon\rightarrow 0$. Once this is achieved, the convergence of $u_{\varepsilon}$ to $u$ would follow by the assumption \eqref{bilip condition on V2}. 
\\
\textbf{Step 1}. Denote $U_{\varepsilon}\coloneqq V_{2}(u_{\varepsilon})$ so that $W_{\varepsilon}=U_{\varepsilon}\ast \eta_{\varepsilon}$. We write an equation governing $U_{\varepsilon}$. Owing to \eqref{nonlocal Cauhy problem} we have 
\begin{align*}
\partial_{t}U_{\varepsilon}=\partial_{t}V_{2}(u_{\varepsilon})&=V'_{2}(u_{\varepsilon})\partial_{t}u_{\varepsilon}=-V'_{2}(u_{\varepsilon})\partial_{x}(V_{1}(U_{\varepsilon}\ast \eta_{\varepsilon})u_{\varepsilon})\\
&=-V'_{2}(u_{\varepsilon})u_{\varepsilon}\partial_{x}(V_{1}(U_{\varepsilon}\ast \eta_{\varepsilon}))-V_{1}(U_{\varepsilon}\ast \eta_{\varepsilon})V'_{2}(u_{\varepsilon})\partial_{x}u_{\varepsilon},    
\end{align*}
which is recast as  
\begin{align}
\partial_
{t}U_{\varepsilon}+V_{2}'(u_\varepsilon)u_{\varepsilon}\partial_{x}(V_{1}(U_{\varepsilon}\ast \eta_{\varepsilon}))+V_{1}(U_{\varepsilon}\ast \eta_{\varepsilon})V_{2}'(u_{\varepsilon})\partial_{x}u_{\varepsilon}=0. \label{eq for Uvareps} 
\end{align}
Similarly, if we denote $U=V_{2}(u)$, then $U$ is governed by 
\begin{align}
\partial_{t}U+V_{2}'(u)u\partial_{x}(V_{1}(U))+V_{1}(U)V_{2}'(u)\partial_{x}u=0.\label{eq for U}
\end{align}
Subtracting \eqref{eq for U} from  \eqref{eq for Uvareps} we get 
\begin{align}
&\partial_{t}(U_{\varepsilon}-U)+V_{2}'(u_{\varepsilon})u_{\varepsilon}\partial_{x}(V_{1}(U_{\varepsilon}\ast \eta_{\varepsilon}))-V_{2}'(u)u\partial_{x}(V_{1}(U)) \notag\\
&+V_{1}(U_{\varepsilon}\ast \eta_{\varepsilon})V_{2}'(u_{\varepsilon})\partial_{x}u_{\varepsilon}-V_{1}(U)V_{2}'(u)\partial_{x}u=0.\label{eq for difference}   \end{align}
Multiplying \eqref{eq for difference} by $\mathbf{s}(t,x)\coloneqq \mathrm{sgn}(U_{\varepsilon}(t,x)-U(t,x))$ we get 
\begin{align}
\partial_{t}\left\vert U_{\varepsilon}-U\right\vert&=\left(V_{2}'(u)u\partial_{x}(V_{1}(U))-V_{2}'(u_{\varepsilon})u_{\varepsilon}\partial_{x}(V_{1}(U_{\varepsilon}\ast \eta_{\varepsilon}))\right)\mathbf{s} \notag\\
&+\left(V_{1}(U)V_{2}'(u)\partial_{x}u-V_{1}(U_{\varepsilon}\ast \eta_{\varepsilon})V_{2}'(u_{\varepsilon})\partial_{x}u_{\varepsilon}\right)\mathbf{s}. \label{eq for abs}
\end{align}
By Lemma \ref{limit of solutions} and continuity of $V_{2}$ we have $\underset{x\rightarrow \pm \infty}{\lim}\left\vert U_{\varepsilon}-U\right\vert(t,x)=0$. Hence, the map $x\mapsto \left\vert U_{\varepsilon}-U\right\vert(t,x)$ attains a maximum at some point $x=x_{\max}(t)$.  
Evaluating \eqref{eq for abs} at $x$ we deduce that 
\begin{align*}
\frac{\dd}{\dd t}\left\Vert (U_{\varepsilon}-U)(t,\cdot)\right\Vert_{\infty}=I+J 
\end{align*}
where we have set 
\begin{align*}
I\coloneqq\left(V_{2}'(u)u\partial_{x}(V_{1}(U))-V_{2}'(u_{\varepsilon})u_{\varepsilon}\partial_{x}(V_{1}(U_{\varepsilon}\ast \eta_{\varepsilon}))\right)\mathbf{s}    
\end{align*}
and 
\begin{align*}
J\coloneqq \left(V_{1}(U)V_{2}'(u)\partial_{x}u-V_{1}(U_{\varepsilon}\ast \eta_{\varepsilon})V_{2}'(u_{\varepsilon})\partial_{x}u_{\varepsilon}\right)\mathbf{s}.   
\end{align*}
Note that the evaluation at the point $(t,x)$ in the definition of $I,J$ is implicit.\\ 
\textbf{Step 2. Estimate on $I$.}  We rewrite $I$ as follows:
\begin{align}
I= (V_{2}'(u)u-V_{2}'(u_{\varepsilon})u_{\varepsilon})\partial_{x}(V_{1}(U))\mathbf{s}+V_{2}'(u_{\varepsilon})u_{\varepsilon}\partial_{x}(V_{1}(U)-V_{1}(U_{\varepsilon}\ast \eta_{\varepsilon}))\mathbf{s}\coloneqq I_{1}+I_{2}. \label{Decompostion of I}   
\end{align}
We start by estimating $I_{1}$. Note that 
\begin{align}
\left\Vert \partial_{x}(V_{1}(U))\right\Vert_{\infty}&=\left\Vert V_{1}'(V_{2}(u))V_{2}'(u)\partial_{x}u\right\Vert_{\infty}
\leq \left\Vert \partial_{x}u\right\Vert_{\infty}\left\Vert V_{1}'\right\Vert_{\infty,\mathrm{loc}}\left\Vert V_{2}'\right\Vert_{\infty,\mathrm{loc}}. \label{first est for I1}  \end{align}
In addition, we have 
\begin{align}
&\left\Vert (V_{2}'(u)u)(t,\cdot)-(V_{2}'(u_{\varepsilon})u_{\varepsilon})(t,\cdot)\right\Vert_{\infty} \notag\\
&\leq   \left\Vert (V_{2}'(u)-V_{2}'(u_{\varepsilon}))(t,\cdot)u(t,\cdot)\right\Vert_{\infty}+\left\Vert V_{2}'(u_{\varepsilon}(t,\cdot))(u_{\varepsilon}-u)(t,\cdot)\right\Vert_{\infty} \notag\\
&\leq \left\Vert V_{2}''\right\Vert_{\infty,\mathrm{loc}}\left\Vert u\right\Vert_{\infty} \left\Vert (u_{\varepsilon}-u)(t,\cdot)\right\Vert_{\infty}+\left\Vert V_{2}'\right\Vert_{\infty,\mathrm{loc}}\left\Vert (u_{\varepsilon}-u)(t,\cdot)\right\Vert_{\infty} \notag\\
&\leq L\left(\left\Vert V_{2}''\right\Vert_{\infty,\mathrm{loc}}\left\Vert u\right\Vert_{\infty}+\left\Vert V_{2}'\right\Vert_{\infty,\mathrm{loc}}\right)\left\Vert (U_{\varepsilon}-U)(t,\cdot)\right\Vert_{\infty}, \label{sec est for I2}
\end{align}
where in the last inequality we used  \eqref{bilip condition on V2}.  Therefore, gathering \eqref{first est for I1}-\eqref{sec est for I2} we deduce that for some constant $C_{1}>0$ it holds that 
\begin{align}
I_{1}\leq  C_{1}\left\Vert (U_{\varepsilon}-U)(t,\cdot)\right\Vert_{\infty}. \label{I1 est}    
\end{align}
We proceed by estimating $I_{2}$. Rewrite $I_{2}$  as  
\begin{align}
I_{2}&=\mathbf{s}V_{2}'(u_{\varepsilon})u_{\varepsilon}\partial_{x}(V_{1}(U)-V_{1}(U\ast \eta_{\varepsilon})) \notag\\
&+\mathbf{s}V_{2}'(u_{\varepsilon})u_{\varepsilon}\partial_{x}(V_{1}(U\ast \eta_{\varepsilon})-V_{1}(U_{\varepsilon}\ast \eta_{\varepsilon}))\coloneqq I_{2}^{1}+I_{2}^{2}. \label{second term in I} 
\end{align} 
By Theorem \ref{Reminder of theorem} it holds that $\left\Vert u_{\varepsilon}\right\Vert_{\infty}
\leq \left\Vert u^{0}\right\Vert_{\infty}$ and so  
\begin{align*}
I_{2}^{1}&\leq \left\Vert V_{2}'\right\Vert_{\infty,\mathrm{loc}}\left\Vert u^{0}\right\Vert_{\infty}\left\Vert (V_{1}'(U)\partial_{x}U-V_{1}'(U\ast \eta_{\varepsilon})\partial_{x}U\ast \eta_{\varepsilon})(t,\cdot)\right\Vert_{\infty}\\
&\leq \left\Vert V_{2}'\right\Vert_{\infty,\mathrm{loc}}\left\Vert u^{0}\right\Vert_{\infty}\left\Vert (V_{1}'(U)-V_{1}'(U\ast \eta_{\varepsilon}))(t,\cdot)\partial_{x}U(t,\cdot)\right\Vert_{\infty}\\
&+\left\Vert V_{2}'\right\Vert_{\infty,\mathrm{loc}}\left\Vert u^{0}\right\Vert_{\infty}\left\Vert V_{1}'(U\ast \eta_{\varepsilon}(t,\cdot))(\partial_{x}U\ast\eta_{\varepsilon}-\partial_{x}U)(t,\cdot)\right\Vert_{\infty}.    
\end{align*}
We can further estimate that 
\begin{align}
&\underset{t\in [0,T]}{\sup}\left\Vert (V_{1}'(U)-V_{1}'(U\ast \eta_{\varepsilon}))(t,\cdot)\partial_{x}U(t,\cdot)\right\Vert_{\infty} \notag\\
&\leq \left\Vert V_{1}''\right\Vert_{\infty,\mathrm{loc}} \left\Vert V_{2}'\right\Vert_{\infty,\mathrm{loc}}\left\Vert \partial_{x}u\right\Vert_{\infty}\underset{t\in [0,T]}{\sup}\left\Vert (U-U\ast \eta_{\varepsilon})(t,\cdot)\right\Vert_{\infty}=O(\varepsilon). \label{ine footnote}
\end{align}
Moreover, we have 
\begin{align*}
&\underset{t\in [0,T]}{\sup}\left\Vert V_{1}'((U\ast \eta_{\varepsilon})(t,\cdot))(\eta_{\varepsilon}\ast \partial_{x}U-\partial_{x}U)(t,\cdot)\right\Vert_{\infty}\\
&\leq \left\Vert V_{1}'\right\Vert_{\infty,\mathrm{loc}}\underset{t\in [0,T]}{\sup}\left\Vert (\partial_{x}U\ast \eta_{\varepsilon}-\partial_{x}U)(t,\cdot)\right\Vert_{\infty}=O(\varepsilon).
\footnotemark
\end{align*}
\footnotetext{
For any $(t,x)$ it holds that: 
\begin{align*}
&\left\vert \partial_{x}U\ast \eta_{\varepsilon}-\partial_{x}U\right\vert (t,x)\leq \frac{1}{\varepsilon}\int_{\mathbb{R}}\left\vert\partial_{x}U(x-y)-\partial_{x}U(x)\right\vert \eta(\frac{y}{\varepsilon})\ \dd y\\
&\leq (\left\Vert V_{2}''\right\Vert_{\infty,\mathrm{loc}}\left\Vert \partial_{x}u\right\Vert_{\infty}^{2}+\left\Vert V_{2}'\right\Vert_{\infty,\mathrm{loc}}\left\Vert \partial_{xx}u\right\Vert_{\infty})\int_{\mathbb{R}}\frac{\left\vert y\right\vert }{\varepsilon}\eta(\frac{y}{\varepsilon}) \ \dd y   \leq  C\varepsilon.    
\end{align*}
The bound \eqref{ine footnote} follows by the same consideration. 
}
This proves that 
\begin{align}
I_{2}^{1}=O(\varepsilon). \label{I12 vanishes asymptotically}    
\end{align}
To estimate $I_{2}^{2}$, we rewrite it as 
\begin{align}
I_{2}^{2}&=\mathbf{s}V_{2}'(u_{\varepsilon})u_{\varepsilon}(V_{1}'(U\ast \eta_{\varepsilon})-V_{1}'(U_{\varepsilon}\ast \eta_{\varepsilon}))\partial_{x}U\ast \eta_{\varepsilon} \notag\\
&+\mathbf{s}V_{2}'(u_{\varepsilon})u_{\varepsilon}V_{1}'(U_{\varepsilon}\ast \eta_{\varepsilon})\partial_{x}(U\ast \eta_{\varepsilon}-U_{\varepsilon}\ast \eta_{\varepsilon}). \label{sec term}
\end{align}
The first term in \eqref{sec term} is bounded by 
\begin{align}
&\left\Vert V_{2}'\right\Vert_{\infty,\mathrm{loc}}\left\Vert u^{0}\right\Vert_{\infty}\left\Vert \partial_{x}U\right\Vert_{\infty}\left\Vert V_{1}''\right\Vert_{\infty,\mathrm{loc}}\left\Vert (U_{\varepsilon}-U)(t,\cdot)\right\Vert_{\infty} \notag\\
&\leq \left\Vert V_{2}'\right\Vert_{\infty,\mathrm{loc}}^{2}\left\Vert u^{0}\right\Vert_{\infty}\left\Vert \partial_{x}u\right\Vert_{\infty}\left\Vert V_{1}''\right\Vert_{\infty,\mathrm{loc}}\left\Vert (U_{\varepsilon}-U)(t,\cdot)\right\Vert_{\infty}.  
\label{first term I22}      \end{align}
As for the second term in \eqref{sec term}, note that \begin{align*}
&\mathbf{s}(t,x)\partial_{x}\left(U\ast \eta_{\varepsilon}-U_{\varepsilon}\ast \eta_{\varepsilon}\right)(t,x)
\\&
=\mathbf{s}(t,x)\left(\frac{1}{\varepsilon^{2}}\int_{x}^{\infty}\eta'(\frac{x-y}{\varepsilon}
)(U-U_{\varepsilon})(t,y)\ \dd y-\frac{\eta(0^{-})(U-U_{\varepsilon})(t,x)}{\varepsilon}\right)\\
&=\frac{\eta(0^{-})}{\varepsilon}\left\Vert (U_{\varepsilon}-U)(t,\cdot)\right\Vert_{\infty}-\frac{\mathbf{s}(t,x)}{\varepsilon^{2}}\int_{x}^{\infty}\eta'(\frac{x-y}{\varepsilon})(U_{\varepsilon}-U)(t,y)
\ \dd y\\
&\geq \frac{\eta(0^{-})}{\varepsilon}\left\Vert (U_{\varepsilon}-U)(t,\cdot)\right\Vert_{\infty}-\frac{\eta(0^{-})}{\varepsilon}\left\Vert (U_{\varepsilon}-U)(t,\cdot)\right\Vert_{\infty}=0.     
\end{align*}
Since $V_{2}'V_{1}'\leq 0$ on $[\lambda_{\min},\lambda_{\max}]$ it follows that the second term in \eqref{sec term} is 
$\leq 0$. Hence, \eqref{first term I22} yields 
\begin{align}
I^{2}_{2}\leq \left\Vert V_{2}'\right\Vert_{\infty,\mathrm{loc}}^{2}\left\Vert u^{0}\right\Vert_{\infty}\left\Vert \partial_{x}u\right\Vert_{\infty}\left\Vert V_{1}''\right\Vert_{\infty,\mathrm{loc}}\left\Vert (U_{\varepsilon}-U)(t,\cdot)\right\Vert_{\infty}.    \label{I22 est}  
\end{align}
Combining \eqref{I12 vanishes asymptotically} with \eqref{I22 est}, we conclude that for some constant $C_{2}>0$ it holds that 
\begin{align}
I_{2}\leq O(\varepsilon)+ C_{2}\left\Vert (U_{\varepsilon}-U)(t,\cdot)\right\Vert_{\infty}. \label{I2 est}     
\end{align}
Combining \eqref{I1 est} with \eqref{I2 est} we have proved that for some constant $C>0$ it holds that 
\begin{align}
I\leq O(\varepsilon)+C\left\Vert (U_{\varepsilon}-U)(t,\cdot)\right\Vert_{\infty}. \label{final est on I}     
\end{align}
\\
\textbf{Step 3. Estimate on $J$.}  Since $x$ is a maximum point of $\left\vert U_{\varepsilon}-U\right\vert(t,\cdot)$ we have $\partial_{x}\left\vert U_{\varepsilon}-U\right\vert(t,x)=0$. Therefore we may rewrite $J$ as 
\begin{align*}
J&=\mathbf{s}(x)(V_{1}(U)-V_{1}(U_{\varepsilon}\ast \eta_{\varepsilon}))\partial_{x}U-\underset{=0}{\underbrace{V_{1}(U_{\varepsilon}\ast \eta_{\varepsilon})\partial_{x}\left\vert U_{\varepsilon}-U\right\vert}}\\
&= \mathbf{s}(x)(V_{1}(U)-V_{1}(U_{\varepsilon}\ast \eta_{\varepsilon}))\partial_{x}U. 
\end{align*}
Consequently we can estimate $J$ as follows.  
\begin{align}
\left\vert J\right\vert 
&\leq \left\Vert \partial_{x}U\right\Vert_{\infty}\left\Vert V_{1}'\right\Vert_{\infty,\mathrm{loc}}\left\Vert (U-U_{\varepsilon}\ast \eta_{\varepsilon})(t,\cdot)\right\Vert_{\infty} \notag\\
&\leq \left\Vert \partial_{x}u\right\Vert_{\infty}\left\Vert V_{2}'\right\Vert_{\infty,\mathrm{loc}} \left\Vert V_{1}'\right\Vert_{\infty,\mathrm{loc}}\left(\left\Vert (U-U\ast \eta_{\varepsilon})(t,\cdot)\right\Vert_{\infty}+\left\Vert(U-U_{\varepsilon})\ast \eta_{\varepsilon}(t,\cdot)\right\Vert_{\infty}  \right) \notag\\
&\leq O(\varepsilon)+\left\Vert \partial_{x}u\right\Vert_{\infty}\left\Vert V_{2}'\right\Vert_{\infty,\mathrm{loc}} \left\Vert V_{1}'\right\Vert_{\infty,\mathrm{loc}}\left\Vert (U_{\varepsilon}-U)(t,\cdot)\right\Vert_{\infty}.    
\label{J expanded}   
\end{align}
To conclude, there is a constant $C'>0$ such that 
\begin{align}
J\leq O(\varepsilon)+C'\left\Vert (U_{\varepsilon}-U)(t,\cdot)\right\Vert_{\infty}. \label{Final est on J}   \end{align} 
\textbf{Step 4. Conclusion.} Combining \eqref{final est on I} with \eqref{Final est on J} we conclude that for some constant $\mathcal{C}>0$ it holds that 
\begin{align*}
\frac{\dd}{\dd t}\left\Vert (U_{\varepsilon}-U)(t,\cdot)\right\Vert_{\infty}\leq O(\varepsilon)+\mathcal{C}\left\Vert (U_{\varepsilon}-U)(t,\cdot)\right\Vert
_{\infty}.
\end{align*}
By Gr\"onwall's inequality we conclude that 
\begin{align*}
\underset{t\in [0,T]}{\sup}\left\Vert (U_{\varepsilon}-U)(t,\cdot)\right\Vert_{\infty}=O(\varepsilon).      
\end{align*}
Thanks to assumption \eqref{bilip condition on V2} on $V_{2}$ we have 
\begin{align*}
\underset{t\in [0,T]}{\sup}\left\Vert (u_{\varepsilon}-u)(t,\cdot)\right\Vert_{\infty} \leq L\underset{t\in [0,T]}{\sup}\left\Vert (U_{\varepsilon}-U)(t,\cdot)\right\Vert_{\infty}.     
\end{align*}
So we  conclude that 
\begin{align}
\underset{t\in [0,T]}{\sup} \left\Vert (u_{\varepsilon}-u)(t,\cdot)\right\Vert_{\infty}=O(\varepsilon)\underset{\varepsilon \rightarrow 0}{\rightarrow} 0. \label{pre final}   \end{align}
\qed
\begin{rem}
We expect that it is possible to relax the assumptions in Theorem \ref{MAIN THM} in the following manner: 
\begin{itemize}
    \item Possibly one can remove the assumption $\eta\in \mathrm{Lip}(\mathbb{R}_{-})$ via an approximation argument similar to the one demonstrated in step 3 of \cite{colombo2023nonlocal}. Since the precise regularity on $\eta$ is not our main focus, we do not elaborate further on this possible extension.
    \item Possibly the assumption $\partial_{x}u^{0}\in \mathrm{BV}(\mathbb{R})$ can be removed by an approximation argument similar to the one in Theorem \ref{approximation lemma}. An argument of this type will be included in the next section, and we so we do not present it here in order to avoid multiplicity.   
\end{itemize}      
\end{rem}
\section{Proof of Theorem \ref{oleinik ine}}\label{Section singular limit}
The proof of Theorem \ref{oleinik ine} is substantially more difficult than the proof of Theorem \ref{MAIN THM}, as it requires identifying rather delicate concavity and monotonicity assumptions on $V_{1},V_{2}$ which would eventually allow to obtain the one sided estimate \eqref{nonlocal oleinik W statement}. In what follows we omit the time variable whenever there is no ambiguity. Set $U_{\varepsilon}\coloneqq V_{2}(u_{\varepsilon})$ so that $W_{\varepsilon}= V_{2}(u_{\varepsilon})\ast \eta_{\varepsilon}$.
In the sequel, we will  constantly use that $W_{\varepsilon}$ is governed by the equation
\begin{align}
\varepsilon\partial_{x}W_{\varepsilon}=W_{\varepsilon}-U_{\varepsilon}. \label{Volterra eq}    
\end{align}
We will first prove \eqref{nonlocal oleinik W statement} for initial data $u^{0}$ in the following narrower class:  
\begin{align}\tag{\textbf{A}}
 u^{0}\in W^{2,\infty}(\mathbb{R})\ \mbox{and}\ \left\Vert u^{0}\right\Vert_{\mathrm{TV}}<\infty. \label{smaller initial data}      
 \end{align}
The case of general $L^{\infty}(\mathbb{R})$ initial data would then follow readily by means of Theorem \ref{approximation lemma}.       
\begin{lem}\label{lemma for W}
Let the assumptions and notations of Theorem \ref{oleinik ine} hold and assume $u^{0}$ satisfies \eqref{smaller initial data}. Then, it holds that  
\begin{align}
\partial_{t}W_{\varepsilon}(x)&+V_{1}(W_{\varepsilon}(x))\partial_{x}W_{\varepsilon}(x) \notag\\
=&-\frac{1}{\varepsilon}\int_{x}^{\infty}\eta(\frac{x-y}{\varepsilon})V'_{1}(W_{\varepsilon}(y))W_{\varepsilon}(y)\partial_{y}W_{\varepsilon}(y)\ \dd y \notag\\
&-\frac{1}{\varepsilon}\int_{x}^{\infty}\eta(\frac{x-y}{\varepsilon})V_{1}'(W_{\varepsilon}(y))(u_{\varepsilon}(y)V_{2}'(u_{\varepsilon}(y))-V_{2}(u_{\varepsilon}(y)))\partial_{y}W_{\varepsilon}(y)\ \dd y. 
\label{equation for W}    
\end{align}
\end{lem}
\begin{proof}
Owing to \eqref{nonlocal Cauhy problem} we have 
\begin{align*}
\partial_{t}U_{\varepsilon}&=V'_{2}(u_{\varepsilon})\partial_{t}u_{\varepsilon}=-V'_{2}(u_{\varepsilon})\partial_{x}(V_{1}(W_{\varepsilon})u_{\varepsilon})\\
&=-V'_{2}(u_{\varepsilon})u_{\varepsilon}\partial_{x}(V_{1}(W_{\varepsilon}))-V_{1}(W_{\varepsilon})V'_{2}(u_{\varepsilon})\partial_{x}u_{\varepsilon}\\
&=-V_{2}'(u_{\varepsilon})u_{\varepsilon}\partial_{x}(V_{1}(W_{\varepsilon}))-V_{1}(W_{\varepsilon})\partial_{x}U_{\varepsilon}. 
\end{align*}
Convolving with $\eta_{\varepsilon}$ we get 
\begin{align}
&\partial_{t}W_{\varepsilon}(x)+V_{1}(W_{\varepsilon}(x))\partial_{x}W_{\varepsilon}(x) \notag\\
&=V_{1}(W_{\varepsilon}(x))\partial_{x}W_{\varepsilon}(x)-\eta_{\varepsilon}\ast (V_{1}(W_{\varepsilon})\partial_{x}U_{\varepsilon})(x)-\eta_{\varepsilon}\ast \left(V_{2}'(u_{\varepsilon})u_{\varepsilon}V_{1}'(W_{\varepsilon})\partial_{x}W_{\varepsilon}\right)(x) \notag\\
&= V_{1}(W_{\varepsilon}(x))\partial_{x}W_{\varepsilon}(x)-\frac{1}{\varepsilon}\int_{x}^{\infty}\eta(\frac{x-y}{\varepsilon})V_{1}(W_{\varepsilon}(y))\partial_{y}U_{\varepsilon}(y)\ \dd y \notag\\
&-\frac{1}{\varepsilon}\int_{x}^{\infty}\eta(\frac{x-y}{\varepsilon})V_{2}'(u_{\varepsilon}(y))u_{\varepsilon}(y)V_{1}'(W_{\varepsilon}(y))\partial_{y}W_{\varepsilon}(y)\ \dd y. \label{first eq for W}
\end{align}
We compute: 
\begin{align}
&V_{1}(W_{\varepsilon}(x))\partial_{x}W_{\varepsilon}(x)-\frac{1}{\varepsilon}\int_{x}^{\infty}\eta(\frac{x-y}{\varepsilon})V_{1}(W_{\varepsilon}(y))\partial_{y}U_{\varepsilon}(y)\ \dd y \notag\\
&=V_{1}(W_{\varepsilon}(x))\partial_{x}W_{\varepsilon}(x)-\frac{1}{\varepsilon}\int_{x}^{\infty}\eta(\frac{x-y}{\varepsilon})V_{1}(W_{\varepsilon}(y))(\partial_{y}W_{\varepsilon}(y)-\varepsilon\partial_{yy}W_{\varepsilon}(y))\ \dd y \notag\\
&=V_{1}(W_{\varepsilon}(x))\partial_{x}W_{\varepsilon}(x)-\frac{1}{\varepsilon}\int_{x}^{\infty}\eta(\frac{x-y}{\varepsilon})V_{1}(W_{\varepsilon}(y))\partial_{y}W_{\varepsilon}(y)\ \dd y \notag\\
&+\int_{x}^{\infty}\eta(\frac{x-y}{\varepsilon})V_{1}(W_{\varepsilon}(y))\partial_{yy}W_{\varepsilon}(y)\ \dd y \notag\\
&\underset{\textup{(integration  by  parts)}}{=}V_{1}(W_{\varepsilon}(x))\partial_{x}W_{\varepsilon}(x)-\frac{1}{\varepsilon}\int_{x}^{\infty}\eta(\frac{x-y}{\varepsilon})V_{1}(W_{\varepsilon}(y))\partial_{y}W_{\varepsilon}(y)\ \dd y \notag\\
&+\frac{1}{\varepsilon}\int_{x}^{\infty}\eta(\frac{x-y}{\varepsilon})V_{1}(W_{\varepsilon}(y))\partial_{y}W_{\varepsilon}(y)\ \dd y-\int_{x}^{\infty}\eta(\frac{x-y}{\varepsilon})V_{1}'(W_{\varepsilon}(y))\left\vert\partial_{y}W_{\varepsilon}\right\vert^{2}(y)\ \dd y \notag\\
&-V_{1}(W_{\varepsilon}(x))\partial_{x}W_{\varepsilon}(x)=-\int_{x}^{\infty}\eta(\frac{x-y}{\varepsilon})V_{1}'(W_{\varepsilon}(y))\left\vert\partial_{y}W_{\varepsilon}\right\vert^{2}(y)\ \dd y. \label{calc}
\end{align}
Using \eqref{Volterra eq}, we see that $\left\vert\partial_{y}W_{\varepsilon}\right\vert^{2}=\frac{1}{\varepsilon}\partial_{y}W_{\varepsilon}(W_{\varepsilon}-U_{\varepsilon})$ and so by  \eqref{calc} we get  
\begin{align}
V_{1}(W_{\varepsilon}(x))\partial_{x}W_{\varepsilon}(x)&-\frac{1}{\varepsilon}\int_{x}^{\infty}\eta(\frac{x-y}{\varepsilon})V_{1}(W_{\varepsilon}(y))\partial_{y}U_{\varepsilon}(y)\ \dd y  \notag\\
=&-\frac{1}{\varepsilon}\int_{x}^{\infty}\eta(\frac{x-y}{\varepsilon})V_{1}'(W_{\varepsilon}(y))\partial_{y}W_{\varepsilon}(y)(W_{\varepsilon}-U_{\varepsilon})(y)\ \dd y \notag\\
=&-\frac{1}{\varepsilon}\int_{x}^{\infty}\eta(\frac{x-y}{\varepsilon})V_{1}'(W_{\varepsilon}(y))W_{\varepsilon}(y)\partial_{y}W_{\varepsilon}(y)\ \dd y \notag\\
&+\frac{1}{\varepsilon}\int_{x}^{\infty}\eta(\frac{x-y}{\varepsilon})V_{1}'(W_{\varepsilon}(y))U_{\varepsilon}(y)\partial_{y}W_{\varepsilon}(y)\ \dd y. \label{identity}    
\end{align}
Substituting \eqref{identity} inside \eqref{first eq for W} we conclude the identity \eqref{equation for W}. 
\end{proof}
\begin{lem} \label{eq for q lemma}
Let the assumptions and notations of Theorem \ref{oleinik ine} hold and assume $u^{0}$ satisfies \eqref{smaller initial data}. Set $q_{\varepsilon}\coloneqq \partial_{x}W_{\varepsilon}$. Then, 
\begin{enumerate}
    \item It holds that
\begin{align}
&\partial_{t}q_{\varepsilon}(x)+V_{1}(W_\varepsilon(x))\partial_{x}q_{\varepsilon}(x) \notag\\
=&-\frac{1}{\varepsilon^{2}}\int_{x}^{\infty}\eta(\frac{x-y}{\varepsilon})V_{1}'(W_{\varepsilon}(y))u_{\varepsilon}(y)V_{2}'(u_{\varepsilon}(y))q_{\varepsilon}(y)\ \dd y \notag\\
&+\frac{V_{1}'(W_{\varepsilon}(x))u_{\varepsilon}(x)V_{2}'(u_{\varepsilon}(x))q_{\varepsilon}(x)}{\varepsilon}-\frac{1}{\varepsilon}\int_{x}^{\infty}\eta(\frac{x-y}{\varepsilon})V_{1}'(W_{\varepsilon}(y))q_{\varepsilon}^{2}(y)\ \dd y.   \label{equation for qepsilon} 
\end{align}
\item For any $t\geq 0$ and any $\varepsilon>0$ it holds that 
\begin{align}
\underset{x\rightarrow \pm \infty}{\lim}q_{\varepsilon}(t,x)=0.  \label{asymptotic vanishing of q statement}    
\end{align}
\end{enumerate}
\end{lem}

\begin{proof}
\textbf{1}. Differentiating in $x$ the equation \eqref{equation for W} we get 
\begin{align*}
&\partial_{t}q_{\varepsilon}(x)+V_{1}(W_{\varepsilon}(x))\partial_{x}q_{\varepsilon}(x)\\
&=-V_{1}'(W_{\varepsilon}(x))q_{\varepsilon}^{2}(x)-\frac{1}{\varepsilon^{2}}\int_{x}^{\infty}\eta(\frac{x-y}{\varepsilon})V_{1}'(W_{\varepsilon}(y))W_{\varepsilon}(y)q_{\varepsilon}(y)\ \dd y\\
&+\frac{V_{1}'(W_{\varepsilon}(x))W_{\varepsilon}(x)q_{\varepsilon}(x)}{\varepsilon}
-\frac{1}{\varepsilon^{2}}\int_{x}^{\infty}\eta(\frac{x-y}{\varepsilon})V_{1}'(W_{\varepsilon}(y))u_{\varepsilon}(y)V_{2}'(u_{\varepsilon}(y))q_{\varepsilon}(y)\ \dd y\\
&+\frac{1}{\varepsilon^{2}}\int_{x}^{\infty}\eta(\frac{x-y}{\varepsilon})V_{1}'(W_{\varepsilon}(y))V_{2}(u_{\varepsilon}(y))q_{\varepsilon}(y)\ \dd y-\frac{1}{\varepsilon}V_{1}'(W_{\varepsilon}(x))V_{2}(u_{\varepsilon}(x))q_{\varepsilon}(x)
\\
&+\frac{1}{\varepsilon}V_{1}'(W_{\varepsilon}(x))u_{\varepsilon}(x)V_{2}'(u_{\varepsilon}(x))q_{\varepsilon}(x)\coloneqq \sum_{k=1}^{7}\mathcal{T}_{k}.   
\end{align*}
We claim that 
\begin{align}
\mathcal{T}_{1}+\mathcal{T}_{2}+\mathcal{T}_{3}+\mathcal{T}_{5}+\mathcal{T}_{6}=-\frac{1}{\varepsilon}\int_{x}^{\infty}\eta(\frac{x-y}{\varepsilon})V_{1}'(W_{\varepsilon}(y))q_{\varepsilon}^{2}(y)\ \dd y. \label{T1T2T3T5T6}  
\end{align}
Note that 
\begin{align*}
\mathcal{T}_{2}&=-\frac{1}{\varepsilon^{2}}\int_{x}^{\infty}\eta(\frac{x-y}{\varepsilon})V_{1}'(W_{\varepsilon}(y))(\varepsilon q_{\varepsilon}+U_{\varepsilon})(y)q_{\varepsilon}(y)\ \dd y\\
&=-\frac{1}{\varepsilon}\int_{x}^{\infty}\eta(\frac{x-y}{\varepsilon})V_{1}'(W_{\varepsilon}(y))q_{\varepsilon}^{2}(y)\ \dd y-\frac{1}{\varepsilon^{2}}\int_{x}^{\infty}\eta(\frac{x-y}{\varepsilon})V_{1}'(W_{\varepsilon}(y))U_{\varepsilon}(y)q_{\varepsilon}(y)\ \dd y.     
\end{align*}
In addition we have 
\begin{align*}
\mathcal{T}_{3}+\mathcal{T}_{6}=\frac{V_{1}'(W_{\varepsilon}(x))(W_{\varepsilon}-U_{\varepsilon})(x)q_{\varepsilon}(x)}{\varepsilon}=\frac{V_{1}'(W_{\varepsilon}(x))\varepsilon q_{\varepsilon}^{2}(x)}{\varepsilon}=V_{1}'(W_{\varepsilon}(x))q_{\varepsilon}^{2}(x),  \end{align*}
so that 
\begin{align}
\mathcal{T}_{1}+\mathcal{T}_{3}+\mathcal{T}_{6}=0. \label{T1T3T6}     
\end{align}
Moreover, it holds that 
\begin{align}
\mathcal{T}_{2}+\mathcal{T}_{5}&=\frac{1}{\varepsilon^{2}}\int_{x}^{\infty}\eta(\frac{x-y}{\varepsilon})V_{1}'(W_{\varepsilon}(y))(U_{\varepsilon}-W_{\varepsilon})(y)q_{\varepsilon}(y)\ \dd y \notag\\
&=-\frac{1}{\varepsilon}\int_{x}^{\infty}\eta(\frac{x-y}{\varepsilon})V_{1}'(W_{\varepsilon}(y))q_{\varepsilon}^{2}(y)\ \dd y. \label{T2T5}    
\end{align}
Combining \eqref{T1T3T6}-\eqref{T2T5} yields \eqref{T1T2T3T5T6}. 
Therefore, it follows that 
\begin{align*}
&\partial_{t}q_{\varepsilon}(x)+V_{1}(W_{\varepsilon}(x))\partial_{x}q_{\varepsilon}(x)\\
=&\mathcal{T}_{4}+\mathcal{T}_{7}-\frac{1}{\varepsilon}\int_{x}^{\infty}\eta(\frac{x-y}{\varepsilon})V_{1}'(W_{\varepsilon}(y))q_{\varepsilon}^{2}(y)\ \dd y\\
=&-\frac{1}{\varepsilon^{2}}\int_{x}^{\infty}\eta(\frac{x-y}{\varepsilon})V_{1}'(W_{\varepsilon}(y))u_{\varepsilon}(y)V_{2}'(u_{\varepsilon}(y))q_{\varepsilon}(y)\ \dd y+\frac{V_{1}'(W_{\varepsilon}(x))u_{\varepsilon}(x)V_{2}'(u_{\varepsilon}(x))q_{\varepsilon}(x)}{\varepsilon}\\
&-\frac{1}{\varepsilon}\int_{x}^{\infty}\eta(\frac{x-y}{\varepsilon})V_{1}'(W_{\varepsilon}(y))q_{\varepsilon}^{2}(y)\ \dd y, 
\end{align*}
as wanted. \\
\textbf{2}. If $\left\Vert u^{0}\right\Vert_{\mathrm{TV}}<\infty$ then $\left\Vert u_{\varepsilon}(t,\cdot)\right\Vert_{\mathrm{TV}}<\infty$. Since any function with finite total variation has a limit at $\pm \infty$, we may write 
\begin{align*}
\underset{x\rightarrow  +\infty}{\lim} u_{\varepsilon}(t,x)=\ell_{\varepsilon}^{+}(t)\ \mbox{and}\ \underset{x\rightarrow -\infty}{\lim}u_{\varepsilon}(t,x)=\ell_{\varepsilon}^{-}(t).     
\end{align*}
In addition, by means of \eqref{Volterra eq} we have 
\begin{align}
\underset{x\rightarrow \pm\infty}{
\lim
}q_{\varepsilon}(t,x)&=\frac{1}{\varepsilon}\underset{x\rightarrow \pm \infty}{\lim}(W_{\varepsilon}(t,x)-U_{\varepsilon}(t,x)) \notag\\
&=\frac{1}{\varepsilon}\underset{x\rightarrow \pm \infty}{\lim}\left(\frac{1}{\varepsilon}\int_{\mathbb{R}}\left(V_{2}(u_{\varepsilon}(t,x-y))-V_{2}(u_{\varepsilon}(t,x))\right)\eta(\frac{y}{\varepsilon})\ \dd y\right). \label{preasymptotic vanishing of q}   
\end{align}
By continuity of $V_{2}$ we have 
\begin{align*}
\underset{z\rightarrow \pm \infty}{\lim}V_{2}(u_{\varepsilon}(t,z))=V_{2}(\ell_{\varepsilon}^{\pm}(t)).    
\end{align*}
Hence, by the dominated 
convergence theorem and \eqref{preasymptotic vanishing of q} we deduce that 
\begin{align*}
\underset{x\rightarrow \pm \infty}{\lim}q_{\varepsilon}(t,x)=0.     
\end{align*}
\end{proof}
In the following Lemma we explain why the assumptions imposed on $V_{1},V_{2}$ imply strict quantitative concavity of the flux function. 
\begin{lem}\label{Concavity of flux function}
Let $V_{1},V_{2}$ satisfy the assumptions of Theorem \ref{oleinik ine} and set $A(\xi)\coloneqq \xi V_{1}(V_{2}(\xi )) $.  Then $A$ is strictly concave on $[0,\mathrm{ess}\sup u^{0}]$. Moreover, it holds that 
\begin{align*}
A''(\xi)\leq -\beta \kappa \ \mbox{for all}\ \xi\in [0,\mathrm{ess}\sup u^{0}],      
\end{align*}
where $\kappa$ is defined as in \eqref{def of kappa}. 
\end{lem}
\begin{proof}
 We compute that 
 \begin{align*}
 (\xi V_{1}(V_{2}(\xi)))'=V_{1}(V_{2}(\xi))+\xi V_{1}'(V_{2}(\xi))V_{2}'(\xi). \end{align*}
Therefore we get 
 \begin{align}
  (\xi V_{1}(V_{2}(\xi)))''&= (V_{1}(V_{2}(\xi))+\xi V_{1}'(V_{2}(\xi))V_{2}'(\xi))' \notag\\
  &=V_{1}'(V_{2}(\xi))V_{2}'(\xi)+V_{1}'(V_{2}(\xi))V_{2}'(\xi)+\xi(V_{1}''(V_{2}(\xi))(V_{2}')^{2}(\xi)+V_{1}'(V_{2}(\xi))V_{2}''(\xi)) \notag\\
&= V'_{2}(\xi)\left(2V_{1}'(V_{2}(\xi))+\xi V_{1}''(V_{2}(\xi))V'_{2}(\xi)+V_{1}'(V_{2}(\xi))\frac{\xi V_{2}''(\xi)}{V'_{2}(\xi)}\right) \notag\\
&=V'_{2}(\xi)\left(V_{1}'(V_{2}(\xi))\left(\frac{V'_{2}(\xi)+\xi V_{2}''(\xi)}{V'_{2}(\xi)}\right)+V_{1}'(V_{2}(\xi))+\xi V_{1}''(V_{2}(\xi))V'_{2}(\xi)\right). \label{double derivative} 
\end{align}
Note that 
\begin{align*}
 &h'(\xi)\leq -\gamma \ \mbox{for all}\ \xi \in I\iff  -\frac{V_{2}'(V_{2}^{-1}(\xi))}{V_{2}'(V_{2}^{-1}(\xi))}-\frac{V_{2}^{-1}(\xi)V_{2}''(V_{2}^{-1}(\xi))}{V_{2}'(V_{2}^{-1}(\xi))}\leq -\gamma\ \mbox{for all}\ \xi \in I.      
\end{align*}
Evaluating the last inequality at $V_{2}(\xi)$ for some $\xi \in [0,\mathrm{ess}\sup u^{0}]$ yields 
\begin{align}
-\frac{V'_{2}(\xi)+\xi V''_{2}(\xi)}{V'_{2}(\xi)}\leq -\gamma \iff \frac{V'_{2}(\xi)+\xi V''_{2}(\xi)}{V'_{2}(\xi)}\geq \gamma. \label{ine in terms of b}      
\end{align}
In view of \eqref{ine in terms of b} and the assumption $V_{1}'\geq \alpha$ on $I$ we get 
\begin{align}
V_{1}'(V_{2}(\xi))\left(\frac{V'_{2}(\xi)+\xi V''_{2}(\xi)}{V'_{2}(\xi)}\right)\geq \alpha \gamma \ \mbox{for all}\ \xi\in [0,\mathrm{ess}\sup u^{0}].  \label{bound alphagamma}   
\end{align}
In addition, for all $\xi\in  [0,\mathrm{ess}\sup u^{0}]$ we have  
\begin{align}
 \xi V_{1}''(V_{2}(\xi))V_{2}'(\xi)=V_{2}^{-1}(V_{2}(\xi))V_{1}''(V_{2}(\xi))V_{2}'(V_{2}^{-1}(V_{2}(\xi)))=-V_{1}''(V_{2}(\xi))h(V_{2}(\xi))\geq -D. \label{bound -D}   
\end{align}
In view of \eqref{bound alphagamma}-\eqref{bound -D}, the term in brackets in \eqref{double derivative} is bounded from below by $\alpha \gamma -D>0$. Since $V'_{2}(\xi)\leq -\beta $ for all $\xi \in [0,\mathrm{ess}\sup u^{0}]$, it follows that 
\begin{align*}
A''(\xi)\leq -\beta(\alpha \gamma -D)=-\beta \kappa \ \mbox{for all}\ \xi\in [0,\mathrm{ess}\sup u^{0}].      
\end{align*}    
\end{proof}
With the above lemmata at hand, we are now well positioned to conclude the proof of Theorem \ref{oleinik ine}.\\ 
\textbf{Proof of Theorem \ref{oleinik ine}.} \textbf{1. A nonlocal Oleinik inequality for $W_{\varepsilon}$.} In this step we prove \eqref{nonlocal oleinik W statement}. As in Lemma \ref{lemma for W} and Lemma \ref{eq for q lemma} we will prove the asserted inequality for $u^{0}$ satisfying \eqref{smaller initial data}. Eventually, this assumption will be easily removed by means of Theorem \ref{approximation lemma}.   
Recall that $q_{\varepsilon}(t,x)=\partial_{x}W_{\varepsilon}(t,x)$ and set $$M(t)\coloneqq \underset{y\in \mathbb{R}}{\sup}\  q_{\varepsilon}(t,y).$$ 
By \eqref{asymptotic vanishing of q statement} it must be that $M(t)\geq 0$ and we may assume with no loss of generality that $M(t)>0$, so that $q_{\varepsilon}(t,\cdot)$ must attain its maximum. Let $x=x_{\max}(t)$ be a point realizing the maximum of $q_{\varepsilon}(t,\cdot)$.  Evaluating \eqref{equation for qepsilon} at $x$ and using that $\partial_{x}q_{\varepsilon}(t,x)=0$ we obtain  
\begin{align}
\frac{\dd}{\dd t}M(t)&=-\frac{1}{\varepsilon^{2}}\int_{x}^{\infty}\eta(\frac{x-y}{\varepsilon})V_{1}'(W_{\varepsilon}(y))u_{\varepsilon}(y)V_{2}'(u_{\varepsilon}(y))q_{\varepsilon}(y)\ \dd y \notag\\
&+\frac{V_{1}'(W_{\varepsilon}(x))u_{\varepsilon}(x)V_{2}'(u_{\varepsilon}(x))M(t)}{\varepsilon}-\frac{1}{\varepsilon}\int_{x}^{\infty}\eta(\frac{x-y}{\varepsilon})V_{1}'(W_{\varepsilon}(y))q_{\varepsilon}^{2}(y)\ \dd y. \label{first eq for M(t)}
\end{align}
Since $-V_{1}'(\xi)V_{2}'(\xi')\geq 0$ for all $(\xi,\xi')\in I\times [0,\mathrm{ess}\sup u
^{0}]$ (by assumption 4. and 6.) we have 
\begin{align}
&-\frac{1}{\varepsilon^{2}}\int_{x}^{\infty}\eta(\frac{x-y}{\varepsilon})V_{1}'(W_{\varepsilon}(y))u_{\varepsilon}(y)V_{2}'(u_{\varepsilon}(y))q_{\varepsilon}(y)\ \dd y \notag\\
&\leq -\frac{M(t)}{\varepsilon^{2}}\int_{x}^{\infty}\eta(\frac{x-y}{\varepsilon})V_{1}'(W_{\varepsilon}(y))u_{\varepsilon}(y)V_{2}'(u_{\varepsilon}(y))\ \dd y. \label{inequality}    
\end{align}
Therefore, substituting \eqref{inequality} inside \eqref{first eq for M(t)} it follows that 
\begin{align}
\frac{\dd}{\dd t}M(t)&\leq \frac{M(t)}{\varepsilon}\left(-\frac{1}{\varepsilon}\int_{x}^{\infty}\eta(\frac{x-y}{\varepsilon})V_{1}'(W_{\varepsilon}(y))u_{\varepsilon}(y)V_{2}'(u_{\varepsilon}(y))\ \dd y+V_{1}'(W_{\varepsilon}(x))u_{\varepsilon}(x)V_{2}'(u_{\varepsilon}(x))\right) \notag\\
&-\frac{1}{\varepsilon}\int_{x}^{\infty}\eta(\frac{x-y}{\varepsilon})V_{1}'(W_{\varepsilon}(y))q_{\varepsilon}^{2}(y)\ \dd y.
\label{first time derivative of M}
\end{align}
Recall that $h(\xi)= -V_{2}^{-1}(\xi)V_{2}'(V_{2}^{-1}(\xi))$ so that $h(U_{\varepsilon})=-u_{\varepsilon}V_{2}'(u_{\varepsilon})$ and $h(\xi)\geq 0$ for all $\xi\in I$. Moreover, $h$ is concave on $I$ by assumption 5. Recall that a concave function $h:I\rightarrow \mathbb{R}$ satisfies the inequality 
\begin{align}
h'(\xi_{1})(\xi_{1}-\xi_{2})\leq h(\xi_{1})-h(\xi_{2}) \leq h'(\xi_{2})(\xi_{1}-\xi_{2})\ \mbox{for all}\ (\xi_{1},\xi_{2})\in I^{2}.    
\end{align}
In view of the last inequality we have 
\begin{align*}
h(U_{\varepsilon})-h(W_{\varepsilon})\leq h'(W_{\varepsilon})(U_{\varepsilon}-W_{\varepsilon})=-\varepsilon h'(W_{\varepsilon})q_{\varepsilon}\Rightarrow h(U_{\varepsilon})\leq h(W_{\varepsilon})-\varepsilon h'(W_{\varepsilon})q_{\varepsilon}.    
\end{align*}
Therefore, since $V_{1}'\geq 0$ on $I$ it follows that 
\begin{align}
&-\frac{1}{\varepsilon}\int_{x}^{\infty}\eta(\frac{x-y}{\varepsilon})V_{1}'(W_{\varepsilon}(y))u_{\varepsilon}(y)V_{2}'(u_{\varepsilon}(y))\ \dd y=\frac{1}{\varepsilon}\int_{x}^{\infty}\eta(\frac{x-y}{\varepsilon})V_{1}'(W_{\varepsilon}(y))h(U_{\varepsilon}(y))\ \dd y \notag\\
&\leq \frac{1}{\varepsilon}\int_{x}^{\infty}\eta(\frac{x-y}{\varepsilon})V_{1}'(W_{\varepsilon}(y))h(W_{\varepsilon}(y))\ \dd y-\int_{x}^{\infty}\eta(\frac{x-y}{\varepsilon})V_{1}'(W_{\varepsilon}(y))h'(W_{\varepsilon}(y))q_{\varepsilon}(y)\ \dd y. 
\label{inequality for integral}
\end{align}
In addition, note the following identity:  
\begin{align*}
&\frac{1}{\varepsilon}\int_{x}^{\infty}\eta(\frac{x-y}{\varepsilon})\left(V_{1}''(W_{\varepsilon}(y))h(W_{\varepsilon}(y))q_{\varepsilon}(y)+V_{1}'(W_{\varepsilon}(y))h'(W_{\varepsilon}(y))q_{\varepsilon}(y)\right)\ \dd y \notag\\
&=\frac{1}{\varepsilon}\int_{x}^{\infty}\eta(\frac{x-y}{\varepsilon})\partial_{y}(V_{1}'(W_{\varepsilon}(y)))h(W_{\varepsilon}(y))\ \dd y+\frac{1}{\varepsilon}\int_{x}^{\infty}\eta(\frac{x-y}{\varepsilon})V_{1}'(W_{\varepsilon}(y))h'(W_{\varepsilon}(y))q_{\varepsilon}(y)\ \dd y
 \notag\\
&\underset{\textup{(integration by parts)}}{=}\frac{1}{\varepsilon^{2}}\int_{x}^{\infty}\eta(\frac{x-y}{\varepsilon})V_{1}'(W_{\varepsilon}(y))h(W_{\varepsilon}(y))\ \dd y-\frac{V_{1}'(W_{\varepsilon}(x))h(W_{\varepsilon}(x))}{\varepsilon},  
\end{align*}
so that 
\begin{align}
&\frac{1}{\varepsilon}\int_{x}^{\infty}\eta(\frac{x-y}{\varepsilon})V_{1}'(W_{\varepsilon}(y))h(W_{\varepsilon}(y))\ \dd y \notag\\
=&\int_{x}^{\infty}\eta(\frac{x-y}{\varepsilon})\left(V_{1}''(W_{\varepsilon}(y))h(W_{\varepsilon}(y))q_{\varepsilon}(y)+V_{1}'(W_{\varepsilon}(y))h'(W_{\varepsilon}(y))q_{\varepsilon}(y)\right)\ \dd y \notag\\
&+V_{1}'(W_{\varepsilon}(x))h(W_{\varepsilon}(x)). \label{identity for noblocal oleinik}   \end{align}
Therefore, substituting \eqref{identity for noblocal oleinik} inside \eqref{inequality for integral} we get 
\begin{align}
&-\frac{1}{\varepsilon}\int_{x}^{\infty}\eta(\frac{x-y}{\varepsilon})V_{1}'(W_{\varepsilon}(y))u_{\varepsilon}(y)V_{2}'(u_{\varepsilon}(y))\ \dd y \notag\\
&\leq \int_{x}^{\infty}\eta(\frac{x-y}{\varepsilon})\left(V_{1}''(W_{\varepsilon}(y))h(W_{\varepsilon}(y))q_{\varepsilon}(y)+V_{1}'(W_{\varepsilon}(y))h'(W_{\varepsilon}(y))q_{\varepsilon}(y)\right)\ \dd y \notag\\
&+V_{1}'(W_{\varepsilon}(x))h(W_{\varepsilon}(x))-\int_{x}^{\infty}\eta(\frac{x-y}{\varepsilon})V_{1}'(W_{\varepsilon}(y))h'(W_{\varepsilon}(y))q_{\varepsilon}(y)\ \dd y \notag\\
&=V_{1}'(W_{\varepsilon}(x))h(W_{\varepsilon}(x))+\int_{x}^{\infty}\eta(\frac{x-y}{\varepsilon})V_{1}''(W_{\varepsilon}(y))h(W_{\varepsilon}(y))q_{\varepsilon}(y)\ \dd y.  \label{prefinal to M inequality}
\end{align}
Putting \eqref{prefinal to M inequality} inside \eqref{first time derivative of M} it follows that 
\begin{align}
&\frac{\dd}{\dd t}M(t) \notag\\
&\leq \frac{M(t)}{\varepsilon}\left(V_{1}'(W_{\varepsilon}(x))(h(W_{\varepsilon}(x))-h(U_{\varepsilon}(x)))+\int_{x}^{\infty}\eta(\frac{x-y}{\varepsilon})V_{1}''(W_{\varepsilon}(y))h(W_{\varepsilon}(y))q_{\varepsilon}(y)\ \dd y\right) \notag\\
&-\frac{1}{\varepsilon}\int_{x}^{\infty}\eta(\frac{x-y}{\varepsilon})V_{1}'(W_{\varepsilon}(y))q_{\varepsilon}^{2}(y)\ \dd y. \label{time derivative of M}
\end{align}
Furthermore, by the mean-value theorem and using the assumption that $h'\leq -\gamma$ on $I$ and that $M(t)\geq 0$ we have the following inequality ($\xi_{0}$ is an intermediate point): 
\begin{align}
h(W_{\varepsilon}(x))-h(U_{\varepsilon}(x))=h'(\xi_{0})(W_{\varepsilon}-U_{\varepsilon})(x)=h'(\xi_{0})\varepsilon\partial_{x}W_{\varepsilon}(x)=h'(\xi_{0})\varepsilon M(t)\leq -\gamma \varepsilon M(t).\label{mean value theorem application}     
\end{align}
Substituting \eqref{mean value theorem application} inside \eqref{time derivative of M} we conclude that 
\begin{align}
\frac{\dd}{\dd t}M(t)\leq& \underset{\leq -\alpha\gamma  M^{2}(t)}{\underbrace{-\gamma V_{1}'(W_{\varepsilon})M^{2}(t)}}+\underset{\leq DM^{2}(t)}{\underbrace{\frac{M(t)}{\varepsilon}\int_{x}^{\infty}\eta(\frac{x-y}{\varepsilon})V_{1}''(W_{\varepsilon}(y))h(W_{\varepsilon}(y))q_{\varepsilon}(y)\ \dd y}} \notag\\
&\underset{\leq 0}{\underbrace{-\frac{1}{\varepsilon}\int_{x}^{\infty}\eta(\frac{x-y}{\varepsilon})V_{1}'(W_{\varepsilon}(y))q_{\varepsilon}^{2}(y)\ \dd y}}
\leq (D-\alpha\gamma) M^{2}(t)=-\kappa M^{2}(t).   \label{Ricatti ine}
\end{align}
Solving inequality \eqref{Ricatti ine} yields 
\begin{align}
M(t)\leq \frac{M(0)}{1+M(0)\kappa t}\leq \frac{1}{\kappa t}. \label{solution of ricatti ine}    
\end{align}
Inequality \eqref{nonlocal oleinik W statement} is now a consequence of \eqref{solution of ricatti ine} and the mean value theorem. 
To remove assumption \eqref{smaller initial data}, we fix a standard mollifier $\chi_{n}$. Given $u^{0}\in L^{\infty}(\mathbb{R})$ set $u^{0}_{n}=\chi_{n}\ast (\mathbf{1}_{[-n,n]}u^{0})$ so that $u^{0}_{n}$ satisfies \eqref{smaller initial data}. Moreover, we have $u^{0}_{n}\underset{n\rightarrow \infty}{\rightarrow} u^{0}$ in $L^{1}_{\mathrm{loc}}(\mathbb{R})$. Denote by $u_{\varepsilon n}$ the solution of \eqref{nonlocal Cauhy problem} with initial data $u^{0}_{n}$ and set $W_{\varepsilon n}=V_{2}(u_{\varepsilon n})\ast \eta_{\varepsilon}$, so that 
\begin{align*}
\frac{W_{\varepsilon n}(t,x)-W_{\varepsilon n}(t,y)}{x-y}\leq \frac{1}{\kappa t}.    
\end{align*}
By Theorem \ref{approximation lemma}, up to a subsequence, we have $W_{\varepsilon n}(t,x)-W_{\varepsilon n}(t,y)\underset{n\rightarrow \infty}{\rightarrow} W_{\varepsilon}(t,x)-W_{\varepsilon}(t,y)$ for a.e. $t>0$ and a.e. $x>y$. Therefore sending $n\rightarrow \infty$ we obtain 
\begin{align*}
\frac{W_{\varepsilon}(t,x)-W_{\varepsilon}(t,y)}{x-y}\leq \frac{1}{\kappa t} \ \mbox{for a.e.} \ t>0\ \mbox{and a.e.}\ x>y . \end{align*}
Since $W_{\varepsilon}$ is continuous, the latter inequality in fact holds for all $t>0$ and all $x>y$.   
\\   
\textbf{2. Convergence of $u_{\varepsilon}$.} 
By equation (4.3) in \cite{bouchut1998one} and \eqref{nonlocal oleinik W statement} we conclude  the following estimate for all compact intervals $K\Subset \mathbb{R}$:  
\begin{align}
\left\Vert W_{\varepsilon}(t,\cdot)\right\Vert_{\mathrm{TV}(K)}\leq 2\left(\frac{\left\vert K\right\vert }{\kappa t}+\left\Vert W_{\varepsilon}\right\Vert_{\infty} \right)\leq 2\left(\frac{\left\vert K\right\vert}{\kappa t}+\underset{\xi \in [0,\mathrm{ess}\sup u^{0}]}{\sup}\left\vert V_{2}\right\vert(\xi)\right). \label{spatial TV}
\end{align}
By means of \eqref{equation for W} we can get a uniform in $\varepsilon$ bound on the total variation in both variables $(t,x)$, namely we claim that for any $\Omega =[\tau,T]\times [-R,R]$ (where $0<\tau<T$ and $R>0$) there is some $C=C(\tau,T,R,\left\Vert u^{0}\right\Vert_{\infty})>0$ such that  
\begin{align}
\left\Vert W_{\varepsilon}\right\Vert_{\mathrm{TV}(\Omega)}\leq C. \label{uniform bound on total variation} 
\end{align}
To prove \eqref{uniform bound on total variation}, it suffices to establish uniform in $\varepsilon$ bounds on the integrals in the right-hand side of \eqref{equation for W}. It holds that 
\begin{align*}
&\frac{1}{\varepsilon}\int_{-R}^{R}\int_{x}^{\infty}\eta(\frac{x-y}{\varepsilon})\left\vert V_{1}'(W_{\varepsilon}(y))W_{\varepsilon}(y) \partial_{y}W_{\varepsilon}(y)\right\vert\ \dd y\dd x\\
&\lesssim \frac{1}{\varepsilon}\int_{-R}^{R}\int_{x}^{\infty}\eta(\frac{x-y}{\varepsilon})\left\vert \partial_{y}W_{\varepsilon}\right\vert(y)\ \dd y\dd x\coloneqq I(t)    
\end{align*}
and 
\begin{align*}
&\frac{1}{\varepsilon}\int_{-R}^{R}\int_{x}^{\infty}\eta(\frac{x-y}{\varepsilon})\left\vert V_{1}'(W_{\varepsilon}(y))\left(u_{\varepsilon}(y)V_{2}'(u_{\varepsilon}(y))-V_{2}(u_{\varepsilon}(y))\right)\partial_{y}W_{\varepsilon}\right\vert (y)\ \dd y\dd x\\
&\lesssim \frac{1}{\varepsilon}\int_{-R}^{R}\int_{x}^{\infty}\eta(\frac{x-y}{\varepsilon})\left\vert \partial_{y}W_{\varepsilon}\right\vert(y)\ \dd y\dd x=I(t).  \end{align*}
We proceed by obtaining uniform in $\varepsilon$ bounds on $I(t)$. First, we decompose $I$ as  
\begin{align*}
I(t)=\frac{1}{\varepsilon}\int_{-R}^{R}\int_{x}^{R}\eta(\frac{x-y}{\varepsilon})\left\vert \partial_{y}W_{\varepsilon}\right\vert(y)\ \dd y\dd x&+\frac{1}{\varepsilon}\int_{-R}^{R}\int_{R}^{\infty}\eta(\frac{x-y}{\varepsilon})\left\vert \partial_{y}W_{\varepsilon}\right\vert(y)\ \dd y\dd x\\
&\coloneqq I_{\leq R}+I_{\geq R}.    
\end{align*}
To bound $I_{\leq R}$ we note that 
\begin{align}
I_{\leq R}&\leq \frac{1}{\varepsilon}\int_{-R}^{R}\int_{-R}^{R}\eta(\frac{x-y}{\varepsilon})\left\vert \partial_{y}W_{\varepsilon}\right\vert(
y)\ \dd y\dd x \notag\\
&\underset{\mathrm{(Fubini)}}{=}\frac{1}{\varepsilon}\int_{-R}^{R}\int_{-R}^{R}\eta(\frac{x-y}{\varepsilon})\left\vert \partial_{y}W_{\varepsilon}\right\vert(y)\ \dd x\dd y\leq \frac{1}{\varepsilon}\int_{-R}^{R}\int_{-\infty}^{\infty}\eta(\frac{x-y}{\varepsilon})\left\vert \partial_{y}W_{\varepsilon}\right\vert(y)\ \dd x\dd y \notag\\
&=\int_{-R}^{R}\left\vert \partial_{y}W_{\varepsilon}\right\vert
(y)\ \dd y\leq 2\left(\frac{2R}{\kappa t}+\underset{\xi \in [0,\mathrm{ess\sup}u^{0}]}{\sup}\left\vert V_{2}\right\vert(\xi) \right), \label{I<R}
\end{align}
where the last inequality is due to \eqref{spatial TV}. 
Furthermore, using \eqref{Volterra eq} we have  \begin{align}
I_{\geq R}&=\frac{1}{\varepsilon^{2}}\int_{-R}^{R}\int_{R}^{\infty}e^{\frac{x-y}{\varepsilon}}\left\vert W_{\varepsilon}-U_{\varepsilon}\right\vert(y)\ \dd y\dd x\lesssim \frac{1}{\varepsilon^{2}}\int_{-R}^{R}\int_{R}^{\infty}e^{\frac{x-y}{\varepsilon}}\ \dd y\dd x \notag\\
&= \left(\frac{1}{\varepsilon}\int_{R}^{\infty}e^{-\frac{\zeta}{\varepsilon}}\ \dd \zeta\right)\left(\frac{1}{\varepsilon}\int_{-R}^{R}e^{\frac{\zeta}{\varepsilon}}\ \dd \zeta\right)=e^{-\frac{R}{\varepsilon}}(e^{\frac{R}{\varepsilon}}-e^{-\frac{R}{\varepsilon}})=1-e^{-\frac{2R}{\varepsilon}}<1. \label{I>R}
\end{align}
Combining \eqref{I<R}-\eqref{I>R} we conclude that 
\begin{align}
I(t)\lesssim \frac{1}{t}+1 \label{bound on I},     
\end{align}
so that 
\begin{align*}
\int_{\tau}^{T}I(t)\ \dd t \lesssim 1.     
\end{align*}
In view of \eqref{equation for W} we conclude \eqref{uniform bound on total variation}. Therefore, by the Fréchet–Kolmogorov theorem there exists some $W\in \mathrm{BV}_{\mathrm{loc}}((0,\infty)\times \mathbb{R})$ such that, up to a subsequence, it holds that $W_{\varepsilon}\underset{\varepsilon \rightarrow 0}{\rightarrow}W$ in $L^{1}_{\mathrm{loc}}((0,\infty)\times \mathbb{R})$. Thanks to the relation $\varepsilon \partial_{x}W_{\varepsilon}=W_{\varepsilon}-U_{\varepsilon}$ and \eqref{uniform bound on total variation} we obtain  
\begin{align*}
 \left\Vert U_{\varepsilon}-W\right\Vert_{L^{1}(\Omega)}\leq \left\Vert W_{\varepsilon}-W\right\Vert_{L^
{1}(\Omega)}+\varepsilon\left\Vert W_{\varepsilon}\right\Vert_{\mathrm{TV}(\Omega)}\underset{\varepsilon \rightarrow 0}{\rightarrow}0.       
\end{align*}
We therefore conclude that $U_{\varepsilon}\underset{\varepsilon \rightarrow 0}{\rightarrow}W$ in $L^{1}_{\mathrm{loc}}((0,\infty)\times \mathbb{R})$, and so up to an extraction of a subsequence we have $U_{\varepsilon}\underset{\varepsilon \rightarrow 0}{\rightarrow} W$ pointwise a.e. By continuity of $V_{2}^{-1}$ and the dominated convergence theorem we infer that   
\begin{align*}
\left\Vert u_{\varepsilon}-V_{2}^{-1}(W)\right\Vert_{L^{1}(\Omega)}=\left\Vert V_{2}^{-1}(U_{\varepsilon})-V_{2}^{-1}(W)\right\Vert_{L^{1}(\Omega)} \underset{\varepsilon \rightarrow 0}{\rightarrow}0.       
\end{align*}
Thus, setting $u\coloneqq V_{2}^{-1}(W)$ we obtain $u_{\varepsilon}\underset{\varepsilon\rightarrow 0}{\rightarrow}u$ in $L^{1}_{\mathrm{loc}}((0,\infty)\times \mathbb{R})$. This immediately implies convergence in $L^{1}_{\mathrm{loc}}([0,\infty)\times \mathbb{R})$ because 
\begin{align*}
\left\Vert u_{\varepsilon}-u\right\Vert_{L^{1}([0,T]\times [-R,R])}&\leq \left\Vert u_{\varepsilon}-u\right\Vert_{L^{1}([0,\tau]\times [-R,R])}+ \left\Vert u_{\varepsilon}-u\right\Vert_{L^{1}([\tau,T]\times [-R,R])}\\
&\leq 4R\tau\left\Vert u^{0}\right\Vert_{\infty}+\left\Vert u_{\varepsilon}-u\right\Vert_{L^{1}([\tau,T]\times [-R,R])}.       
\end{align*}
Letting $\varepsilon \rightarrow 0$ and then letting $\tau\rightarrow 0$  in the last inequality implies $u_{\varepsilon}\underset{\varepsilon \rightarrow 0}{\rightarrow} u$ in $L^{1}_{\mathrm{loc}}([0,\infty)\times \mathbb{R})$. 
Now, passing to the limit as $\varepsilon\rightarrow 0$ in \eqref{nonlocal oleinik W statement} we get 
\begin{align*}
 \frac{V_{2}(u(t,x))-V_{2}(u(t,y))}{x-y}\leq \frac{1}{\kappa t}.    
\end{align*}
We therefore conclude that 
\begin{align}
 \frac{u(t,x)-u(t,y)}{x-y}\geq \frac{1}{\underset{\xi\in [0,\mathrm{ess}\sup u^{0}]}{\sup}V_{2}'(\xi)\kappa t}\geq -\frac{1}{\beta \kappa t}. \label{one sided lip for u}
 \end{align}
We proceed by showing that $u$ is a weak solution. Thanks to Theorem \ref{equiv chara}, Lemma \ref{Concavity of flux function} and \eqref{one sided lip for u} this would imply that $u$ is the unique entropy solution. Given a test function $\varphi\in C^{\infty}_{0}((-a,T)\times \mathbb{R})$ we have 
\begin{align*}
&\iint_{[0,T]\times \mathbb{R}}u_{\varepsilon}(\tau,x)\partial_{\tau}\varphi(\tau,x)\ \dd x\dd \tau+\int_{\mathbb{R}}u^{0}(x)\varphi(0,x)\ \dd x\\
&+\iint_{[0,T]\times \mathbb{R}}V_{1}(W_{\varepsilon}(\tau,x))u_{\varepsilon}(\tau,x)\partial_{x}\varphi(\tau,x) \ \dd x\dd \tau=0.     
\end{align*}
Clearly, we have 
\begin{align*}
\iint_{[0,T]\times \mathbb{R}}u_{\varepsilon}(\tau,x)\partial_{\tau}\varphi(\tau,x)\ \dd x\dd \tau\underset{\varepsilon \rightarrow 0}{\rightarrow}\iint_{[0,T]\times \mathbb{R}}u(\tau,x)\partial_{\tau}\varphi(\tau,x)\ \dd x\dd \tau.     
\end{align*}
In addition, we have 
\begin{align}
&\iint_{[0,T]\times \mathbb{R}}V_{1}(W_{\varepsilon}(\tau,x))u_{\varepsilon}(\tau,x)\partial_{x}\varphi(\tau,x)\ \dd x\dd \tau-\iint_{[0,T]\times \mathbb{R}}V_{1}(W(\tau,x))u(\tau,x)\partial_{x}\varphi(\tau,x)\ \dd x\dd \tau \notag\\
&=\iint_{[0,T]\times \mathbb{R}}(V_{1}(W_{\varepsilon}(\tau,x))-V_{1}(W(\tau,x)))u_{\varepsilon}(\tau,x)\partial_{x}\varphi(\tau,x)\ \dd x\dd \tau \notag\\
&+\iint_{[0,T]\times \mathbb{R}}V_{1}(W(\tau,x))(u_{\varepsilon}-u)(\tau,x)\partial_{x}\varphi(\tau,x)\ \dd x\dd \tau. \label{second weak integral}
\end{align}
Setting $\Omega=\mathrm{supp}(\varphi)$, the first integral in the right-hand side of \eqref{second weak integral} is bounded by 
\begin{align*}
&\left\Vert u_{\varepsilon}\right\Vert_{\infty}\left\Vert V_{1}'\right\Vert_{\infty,\mathrm{loc}}\left\Vert \partial_{x}\varphi\right\Vert_{\infty} \left\Vert W_{\varepsilon}-W\right\Vert_{L^{1}(\Omega)}\\
&\leq \left\Vert u^{0}\right\Vert_{\infty}\left\Vert V_{1}'\right\Vert_{\infty,\mathrm{loc}}\left\Vert \partial_{x}\varphi\right\Vert_{\infty}\left\Vert W_{\varepsilon}-W\right\Vert_{L^{1}(\Omega)}\underset{\varepsilon \rightarrow 0}{\rightarrow}0.           
\end{align*}
The second integral in the right-hand side of \eqref{second weak integral} is bounded by 
\begin{align*}
\left\Vert V_{1}\right\Vert_{\infty,\mathrm{loc}}\left\Vert \partial_{
x}\varphi\right\Vert_{\infty} \left\Vert u_{\varepsilon}-u\right\Vert_{L^{1}(\Omega)}\underset{\varepsilon\rightarrow 0}{\rightarrow} 0.       
\end{align*}
Finally, by uniqueness of entropy solutions the convergence in fact holds along the entire sequence. 
\qed

\vspace{0.4 cm}

\begin{rem}\label{example of V1 V2}
We exemplify nontrivial (i.e. non-affine) functions $V_{1},V_{2}$ verifying the assumptions of Theorem \ref{oleinik ine}. Consider $V_{1}(\xi)=\xi^{2}$, $V_{2}(\xi)=e^{-\xi}$ and assume further that $u^{0}\in L^{\infty}(\mathbb{R}), u^{0}\geq 0$ satisfies $\mathrm{ess}\sup u^{0}<\frac{1}{2}$. We verify that conditions 4.-6. are satisfied:  
\begin{itemize}
    \item $V_{2}'(\xi)\leq -e^{-\mathrm{ess}\sup u^{0}}=-\beta$ for all $\xi \in [0,\mathrm{ess}\sup u^{0}]$. 
    \item $V_{2}^{-1}(\xi)=-\log \xi$ so that $h(\xi)=-\xi\log \xi$. Therefore $h'(\xi)=-\log \xi-1$ and $h''(\xi)=-\frac{1}{\xi}\leq 0$. Moreover, $I=\left[e^{-\mathrm{ess }\sup u^{0}},1\right]$. Therefore, in this case $$h'(\xi)\leq \mathrm{ess}\sup u^{0}-1=-\gamma \ \mbox{for all}\ \xi\in I.$$  
    \item $V_{1}'(\xi)=2\xi\geq 2e^{-\mathrm{ess}\sup u^{0}}=\alpha$ for all $\xi\in [e^{-\mathrm{ess\sup u^{0}}},1] $ and  $V_{1}''(\xi)=2\geq0$. In addition, we have  $V_{1}''(\xi)h(\xi)=-2\xi\log \xi$. Note also that $(-2\xi \log(\xi))'=-2\log(\xi)-2\leq 0$ on $I$ provided $\mathrm{ess}\sup u^{0}<1$ so that $\xi\mapsto V_{1}''(\xi)h(\xi)$ is non-increasing on $I$. Therefore, we get   
\begin{align*}
D=\underset{\xi \in I}{\sup}V_{1}''(\xi)h(\xi)&=V_{1}''(\xi)h(\xi)\vert_{\xi=e^{-\mathrm{ess}\sup u^{0}}}\\
&=2\mathrm{ess}\sup u^{0}e^{-\mathrm{ess}\sup u^{0}}<2e^{-\mathrm{ess}\sup u^{0}}(1-\mathrm{ess}\sup u^{0})=\alpha \gamma,      \end{align*}
where the last inequality is due to the assumption $\mathrm{ess}\sup u^{0}<\frac{1}{2}$. 
\end{itemize}
\end{rem}
\vspace{0.4 cm}
\noindent{\bf Acknowledgments.} The author would like to thank Gianluca Crippa and Laura Spinolo for their interest in this work and for several stimulating discussions. The author is indebted to the financial support granted by the University of Basel. 

\bigskip

\noindent{\bf Conflict of Interest:} The author declares that he has no conflict of interest.
\bigskip

\noindent{\bf Data availability:} Data sharing is not applicable to this article as no datasets were generated or analyzed during the current study.

\bigskip

\noindent{\bf Declaration on the Use of Generative AI.} The author acknowledges the use of ChatGPT as an auxiliary tool in exploring mathematical ideas and checking certain intermediate arguments and computations. The main use of AI is in \S\ref{Section singular limit}, where some of the estimates enabling to close the Ricatti inequality for $M$ were discovered via back and forth correspondence with ChatGPT. This procedure eventually lead to the identification of assumptions 3.-6. in Theorem \ref{oleinik ine}, which play a decisive role in the proof of the non-local Oleinik inequality. In addition, the author used ChatGPT in order to verify the correctness and accuaracy of the final results.  All AI-generated suggestions were independently assessed and verified by the author. The mathematical results, proofs, organization, and conclusions of the paper are the author’s own, and the manuscript was written in its entirety by the author, who assumes full responsibility for its correctness. 
\bibliographystyle{abbrv}
\bibliography{references}

\appendix 
\section{}
\textbf{Proof of Theorem \ref{Well posedness of nonlocal eq}}. Considering that $\varepsilon$ will not play an essential role, we take $\varepsilon=1$ in the course of the proof. \\ \textbf{Existence}. \textbf{1.} Given arbitrary $T>0$ we will construct a weak solution $u\in L^{\infty}([0,T]\times \mathbb{R})$. Let $\chi_{n}$ be a standard mollifier and set $u^{0}_{n}=\chi_{n}\ast (\mathbf{1}_{[-n,n]}u^{0})$. By Theorem \ref{Reminder of theorem} there is a unique solution $u_{n}$ to the regularized equation 
\begin{align}
\begin{cases}
\partial_{t}u_{n}(t,x)+\partial_{x}(V_{1}(V_{2}(u_{n})\ast \eta)u_{n}(t,x))=0, \ (t,x)\in (0,\infty)\times \mathbb{R} \\ u(0,x)=u^{0}_{n}, \ x\in \mathbb{R}.  \end{cases}
\label{nonlocal Cauhy problem appendix}  
\end{align}
Moreover, one has the uniform bound in $n$: 
\begin{align*}
 0\leq u_{n} \leq \mathrm{ess}\sup u
 ^{0}_{n}\leq \mathrm{ess}\sup u^{0}.
\end{align*}
Let $u\in L^{\infty}([0,T]\times \mathbb{R})$ be such that $u_{n}\underset{ n\rightarrow \infty}{\rightharpoonup^{\ast}} u$ in $L^{\infty}([0,T]\times \mathbb{R})$ (up to an extraction of a subsequence). Set $U_{n}\coloneqq V_{2}(u_{n})$ and $W_{n}\coloneqq V_{2}(u_{n})\ast \eta$. In the sequel $C>0$ designates any "harmless" constant (in particular independent of $n$) that may vary between the different equations. In addition, we will not specify explicitly when extracting a subsequence. Note the following: 
\begin{align}
\left\Vert W_{n}\right\Vert_{\infty}\leq \underset{\xi\in [0,\mathrm{ess}\sup u^{0}]}{\sup}\left\vert V_{2}\right\vert(\xi)\leq C,\ \left\Vert \partial_{x}W_{n}\right\Vert_{\infty}\leq 2\underset{\xi\in [0,\mathrm{ess}\sup u^{0}]}{\sup}\left\vert V_{2}\right\vert(\xi)\leq C. \label{uniform in n bound}          
\end{align}
By means of \eqref{equation for W} we also get $\left\Vert \partial_{t}W_{n}\right\Vert_{\infty}\leq C$. Therefore there exist some $W\in W^{1,\infty}([0,T]\times \mathbb{R})$ such that $W_{n}\underset{n\rightarrow \infty}{\rightarrow}W$ in $L^{\infty}_{\mathrm{loc}}([0,T]\times \mathbb{R})$. Note that at the moment we still do not know that $W=V_{2}(u)\ast \eta$. Denote by $X_{n},X$ characteristic lines associated with $V_{1}(W_{n})$ and $V_{1}(W)$ respectively, i.e. the solutions to   
\begin{align*}
\frac{\dd}{\dd t}X_{n}(t,x)=V_{1}(W_{n}(t,X_{n}(t,x))),\ X_{n}(0,x)=x,     
\end{align*}
and 
\begin{align*}
\frac{\dd}{\dd t}X(t,x)=V_{1}(W(t,X(t,x))),\ X(0,x)=x.     
\end{align*}
Subtracting the second equation from the first we obtain 
\begin{align*}
&\frac{\dd}{\dd t}(X_{n}(t,x)-X(t,x))\\
&=V_{1}(W_{n}(t,X_{n}(t,x)))-V_{1}(W(t,X(t,x)))\\
&=V_{1}(W_{n}(t,X_{n}(t,x)))-V_{1}(W_{n}(t,X(t,x)))+V_{1}(W_{n}(t,X(t,x)))-V_{1}(W(t,X(t,x))).
\end{align*}
In what follows $K\Subset \mathbb{R}$ is a compact interval of $\mathbb{R}$. 
\begin{align*}
\left\Vert (X_{n}-X)(t,\cdot)\right\Vert_{L^{\infty}(K)}&\leq \left\Vert V_{1}'\right\Vert_{\infty,\mathrm{loc}} \int_{0}^{t}\left\Vert W_{n}(\tau,X_{n}(\tau,\cdot ))-W_{n}(\tau,X(\tau,\cdot))\right\Vert_{L^{\infty}(K)}\ \dd \tau\\
&+\left\Vert V_{1}'\right\Vert_{\infty,\mathrm{loc}}\int_{0}^{t}\left\Vert (W_{n}-W)(\tau,X(\tau,\cdot))\right\Vert_{L^{\infty}(K)}\ \dd \tau\\
&\leq C\left(\int_{0}^{t}\left\Vert (X_{n}-X)(\tau,\cdot)\right\Vert_{L^{\infty}(K)}\ \dd \tau+o_{n}(1) \right).
\end{align*}
In the last inequality we used that $\left\Vert \partial_{x}W_{n}\right\Vert_{\infty}$ is uniformly bounded in $n$. 
Thus, it follows that $X_{n}\underset{n\rightarrow \infty}{\rightarrow}X$ in $L^{\infty}_{\mathrm{loc}}([0,T]\times \mathbb{R})$. From the same considerations we get $X_{n}^{-1}\underset{n\rightarrow \infty}{\rightarrow}X^{-1}$ in $L^{\infty}_{\mathrm{loc}}([0,T]\times \mathbb{R})$. \\
\textbf{2}. In this step we show that $u_{n}\underset{n\rightarrow \infty}{\rightarrow}u$ in $L^{1}_{\mathrm{loc}}([0,T]\times \mathbb{R})$, which constitutes the main part of the existence proof. 
We apply the following general representation formula (see for example equation (2.11) in \cite{colombo2024multidimensional}):
\begin{align*}
u_{n}(t,X_{n}(t,x))=u^{0}_{n}(x)\exp\left(-\int_{0}^{t}\partial_{2}(V_{1}(W_{n}))(\tau,X_{n}(\tau,x))\ \dd \tau\right).    
\end{align*}
First we claim that $\{u_{n}(t,X_{n}(t,x))\}_{n\geq 1}$ is a Cauchy sequence in $L^{1}([0,T]\times K)$. Given $n,m\geq 1$ we have  
\begin{align*}
&\left\Vert u_{n}(t,X_{n}(t,\cdot))-u_{m}(t,X_{m}(t,\cdot))\right\Vert_{L^{1}(K)}\\
&\leq C\left\Vert u^{0}_{n}-u^{0}_{m}\right\Vert_{L^{1}(K)} 
\\&+\left\Vert u^{0}\right\Vert_{\infty}\left\Vert \exp\left(-\int_{0}^{t}\partial_{2}(V_{1}(W_{n}))(\tau,X_{n}(\tau,\cdot))\ \dd \tau\right)-\exp\left(-\int_{0}^{t}\partial_{2}(V_{1}(W_{m}))(\tau,X_{m}(\tau,\cdot))\ \dd\tau\right)\right\Vert_{L^{1}(K)}.  
\end{align*}
The second term in the above inequality is bounded by 
\begin{align*}
& C\int_{0}^{t}\left\Vert V_{1}'(W_{n})(W_{n}-U_{n})(\tau,X_{n}(\tau,\cdot))-V_{1}'(W_{m})(W_{m}-U_{m})(\tau,X_{m}(\tau,\cdot))\right\Vert_{L^{1}(K)}\ \dd\tau\\
&\leq C\int_{0}^{t}\left\Vert \left(V_{1}'(W_{n}(\tau,X_{n}(\tau,\cdot)))-V_{1}'(W_{m}(\tau,X_{m}(\tau,\cdot)))\right)(W_{n}-U_{n})(\tau,X_{n}(\tau,\cdot))\right\Vert_{L^{1}(K)}\ \dd \tau\\
&+C\int_{0}^{t}\left\Vert V_{1}'(W_{m}(\tau,X_{m}(\tau,\cdot)))\left(W_{n}(\tau,X_{n}(\tau,\cdot))-W_{m}(\tau,X_{m}
(\tau,\cdot))\right)\right\Vert_{L^{1}(K)}\ \dd \tau\\
&+C\int_{0}^{t}\left\Vert V_{1}'(W_{m}(\tau,X_{m}(\tau,\cdot)))\left(U_{m}(\tau,X_{m}(\tau,\cdot))-U_{n}(\tau,X_{n}(\tau,\cdot))\right)\right\Vert_{L^{1}(K)}\ \dd \tau
\\&\leq C\int_{0}^{t}\left\Vert W_{n}(\tau,X_{n}(\tau,\cdot))-W_{m}(\tau,X_{m}(\tau,\cdot))\right\Vert_{L^{1}(K)}\ \dd\tau\\
&+C\int_{0}^{t}\left\Vert U_{m}(\tau,X_{m}(\tau,\cdot))-U_{n}(\tau,X_{n}(\tau,\cdot))\right\Vert_{L^{1}(K)}\ \dd \tau \coloneqq J_{1}+J_{2}.
\end{align*}
First, observe that $J_{2}$ can be estimated as 
\begin{align*}
J_{2}\leq& \left\Vert V_{2}'\right\Vert_{\infty,\mathrm{loc}}\int_{0}^{t}\left\Vert u_{m}(\tau,X_{m}(\tau,\cdot))-u_{n}(\tau,X_{n}(\tau,\cdot))\right\Vert_{L^{1}(K)}\ \dd \tau\\
\leq& C\int_{0}^{t}\left\Vert u_{m}(\tau,X_{m}(\tau,\cdot))-u_{n}(\tau,X_{n}(\tau,\cdot))\right\Vert_{L^{1}(K)}\ \dd \tau.      
\end{align*}
As for $J_{1}$, start by noting that for any $k\geq 1$ we have  
\begin{align*}
&\left\Vert W_{k}(t,X_{k}(t,\cdot))-W(t,X(t,\cdot))\right\Vert_{L^{1}(K)}\\ &\leq \left\Vert W_{k}(t,X_{k}(t,\cdot))-W_{k}(t,X(t,\cdot))\right\Vert_{L^{1}(K)}+\left\Vert W_{k}(t,X(t,\cdot))-W(t,X(t,\cdot))\right\Vert_{L^{1}(K)}\\
&\leq \left\Vert \partial_{x}W_{k}\right\Vert_{\infty}\left\Vert X_{k}(t,\cdot)-X(t,\cdot)\right\Vert_{L^{1}(K)}+\left\Vert W_{k}(t,X(t,\cdot))-W(t,X(t,\cdot))\right\Vert_{L^{1}(K)}
\\&\underset{\eqref{uniform in n bound}}{\leq} C\left\Vert X_{k}(t,\cdot)-X(t,\cdot)\right\Vert_{L^{1}(K)}+\left\Vert W_{k}(t,X(t,\cdot))-W(t,X(t,\cdot))\right\Vert_{L^{1}(K)}\underset{k\rightarrow \infty}{\rightarrow} 0.    
\end{align*}
Therefore, given $\eta>0$,  there is some $n_{0}$ such that for all $n,m\geq n_{0}$ it holds that $J_{1}\leq \eta$. Consequently, we conclude that there is some $n_{0}'$ such that for all $n,m\geq n_{0}'$ it holds that 
\begin{align*}
&\left\Vert u_{n}(t,X_{n}(t,\cdot))-u_{m}(t,X_{m}(t,\cdot))\right\Vert_{L^{1}(K)}\\
&\leq \eta+C\int_{0}^{t}\left\Vert u_{n}(\tau,X_{n}(\tau,\cdot))-u_{m}(\tau,X_{m}(\tau,\cdot))\right\Vert_{L^{1}(K)} \dd \tau,      
\end{align*}
and hence 
\begin{align*}
 \left\Vert u(t,X_{n}(t,\cdot))-u_{m}(t,X_{m}(t,\cdot))\right\Vert_{L^{1}(K)}\leq C\eta.    
\end{align*}
So $\left\{u_{n}(t,X_{n}(t,x))\right\}_{n\geq 1}$ is a Cauchy sequence in $L^{1}([0,T]\times K)$ and therefore we conclude that there is some $\bar{u}\in L^{1}_{\mathrm{loc}}([0,T]\times \mathbb{R})$ such that $u_{n}(t,X_{n}(t,\cdot))\underset{n\rightarrow \infty}{\rightarrow} \bar{u}$ in $L^{1}_{\mathrm{loc}}([0,T]\times \mathbb{R})$. Next, we assert that $u_{n}(t,x)\underset{n\rightarrow \infty}{\rightarrow}\bar{u}(t,X^{-1}(t,x))$ in $L^{1}_{\mathrm{loc}}([0,T]\times \mathbb{R})$. 
 By the triangle inequality we have 
\begin{align}
&\left\Vert u_{n}(t,\cdot)-\bar{u}(t, X^{-1}(t,\cdot))\right\Vert_{L^{1}(K)} \notag\\ 
&\leq \left\Vert u_{n}(t,X_{n}\circ X_{n}^{-1}(t,\cdot))-\bar{u}(t,X_{n}^{-1}(t,\cdot)) \right\Vert_{L^{1}(K)}+\left\Vert \bar{u}(t,X_{n}^{-1}(t,\cdot))-\bar{u}(t,X^{-1}(t,\cdot))\right\Vert_{L^{1}(K)}. \label{triangle} 
\end{align}
To estimate the first term in \eqref{triangle} we apply the change of variables $z=X_{n}^{-1}(t,y)$ in order to obtain
\begin{align*}
&\left\Vert u_{n}(t,X_{n}\circ X_{n}^{-1}(t,\cdot))-\bar{u}(t,X_{n}^{-1}(t,\cdot)) \right\Vert_{L^{1}(K)}\\
&=\int_{X_{n}^{-1}(t,K)} \left\vert u_{n}(t,X_{n}(t,z))-\bar{u}(t,z)\right\vert \partial_{z}X_{n}(t,z) \ \dd z.      
\end{align*}
Since $X_{n}^{-1}(t,\cdot)$ is a  diffeomorphism, it transfers compact sets to compact sets. Moreover there is some compact $K'\Subset \mathbb{R}$ such that $X_{n}^{-1}(t,K)\subset K'$ for all $t\in [0,T]$ and all $n\geq1$. Therefore, using that $\left\Vert \partial_{z}X_{n}\right\Vert_{\infty}\leq C $ we find that    
\begin{align}
\left\Vert u_{n}(t,X_{n}\circ X_{n}^{-1}(t,\cdot))-\bar{u}(t,X_{n}^{-1}(t,\cdot)) \right\Vert_{L^{1}(K)}\leq C\int_{K'}\left\vert u_{n}(t,X_{n}(t,z))-\bar{u}(t,z)\right\vert\ \dd z\underset{n\rightarrow \infty}{\rightarrow} 0. \label{firstvanishing}  \end{align}
To estimate the second term in \eqref{triangle} we apply the change of variables $z=X^{-1}(t,y)$ to find 
\begin{align*}
&\left\Vert \bar{u}(t,X_{n}^{-1}(t,\cdot))-\bar{u}(t,X^{-1}(t,\cdot))\right\Vert_{L^{1}(K)}\\
&=\int_{X^{-1}(t,K)}\left\vert \bar{u}(t,X_{n}^{-1}(t,X(t,z)))-\bar{u}(t,z)\right\vert\partial_{z}X(t,z)\ \dd z.      
\end{align*}
As before, there is some compact set $K'\Subset \mathbb{R}$ such that $X^{-1}(t,K)\subset K'$ for all $t\in [0,T]$. Therefore, using that $\left\Vert \partial_{z}X\right\Vert_{\infty}\leq C$ we find that 
\begin{align*}
\left\Vert \bar{u}(t,X_{n}^{-1}(t,\cdot))-\bar{u}(t,X^{-1}(t,\cdot))\right\Vert_{L^{1}(K)}\leq C\int_{K'}\left\vert \bar{u}(t,X_{n}^{-1}(t,X(t,z)))-\bar{u}(t,z)\right\vert\ \dd z.     
\end{align*}
Now fix a mollifier $\chi_{\lambda}$ and set $\bar{u}_{\lambda}\coloneqq\chi_{\lambda}\ast \bar{u}$. In view of the last inequality we get 
\begin{align}
&\left\Vert \bar{u}(t,X_{n}^{-1}(t,\cdot))-\bar{u}(t,X^{-1}(t,\cdot))\right\Vert_{L^{1}(K)} \notag\\
&\leq     C\int_{K'}\left\vert \bar{u}_{\lambda}(t,X_{n}^{-1}(X(t,z)))-\bar{u}_{\lambda}(t,z)\right\vert\ \dd z \notag\\
&+C\left(\int_{K'}\left\vert \bar{u}_{\lambda}-\bar{u}\right\vert(t,X_{n}^{-1}(X(t,z)))\ \dd z+\int_{K'}\left\vert \bar{u}_{\lambda}-\bar{u}\right\vert(t,z)\ \dd z\right).
\label{3 integrals}
\end{align}
Since $\bar{u}_{\lambda}$ is continous and $X_{n}^{-1}\circ X(t,z)\underset{n\rightarrow \infty}{\rightarrow} z$ pointwise, it follows that $\bar{u}_{\lambda}(t,X_{n}^{-1}(X(t,z)))\underset{n\rightarrow \infty}{\rightarrow} \bar{u}_{\lambda}(t,z)$ pointwise. Thus, by the dominated convergence theorem we conclude that
\begin{align*}
\int_{K'}\left\vert \bar{u}_{\lambda}(t,X_{n}^{-1}(X(t,z)))-\bar{u}_{\lambda}(t,z)\right\vert\ \dd z\underset{n\rightarrow \infty}{\rightarrow} 0.   
\end{align*}
In addition, by the change of variables $\zeta=X_{n}^{-1}\circ X$ we have 
\begin{align*}
\int_{K'}\left\vert \bar{u}_{\lambda}-\bar{u}\right\vert(t,X_{n}^{-1}(X(t,z)))\ \dd z\leq C\int_{K''}\left\vert \bar{u}_{\lambda}-\bar{u}\right\vert(t,\zeta)\ \dd \zeta\underset{\lambda \rightarrow 0}{\rightarrow} 0\  \mbox{for\ some\ $K''\Subset \mathbb{R}$}.      
\end{align*}
Hence the two last integrals in \eqref{3 integrals} vanish as $\lambda\rightarrow 0$ uniformly in $n$. So sending $n\rightarrow \infty$ and then $\lambda\rightarrow 0$ in \eqref{3 integrals} we find that 
\begin{align}
\left\Vert \bar{u}(t,X_{n}^{-1}(t,\cdot))-\bar{u}(t,X^{-1}(t,\cdot))\right\Vert_{L^{1}(K)}\underset{n\rightarrow \infty}{\rightarrow} 0. \label{secondvanishing} \end{align}
Substituting \eqref{firstvanishing} and \eqref{secondvanishing} in \eqref{triangle} we conclude that $u_{n}(t,x)\underset{n\rightarrow \infty}{\rightarrow} \bar{u}(t,X^{-1}(t,x))$ in $L^{1}_{\mathrm{loc}}([0,T]\times \mathbb{R})$. Finally, by uniqueness of the limit we immediately infer: 
\begin{itemize}
    \item $u(t,x)\equiv\bar{u}(t,X^{-1}(t,x))$ and $u_{n}\underset{n\rightarrow \infty}{\rightarrow} u$ in $L^{1}_{\mathrm{loc}}([0,T]\times \mathbb{R})$. 
    \item $W=V_{2}(u)\ast \eta$ and $W_{n}\underset{n\rightarrow \infty}{\rightarrow} V_{2}(u)\ast \eta$ in $L^{\infty}_{\mathrm{loc}}([0,T]\times \mathbb{R})$. 
\end{itemize}
\textbf{3}. The proof that $u$ is the requested weak solution is now easily concluded. Given a test function $\varphi\in C^{\infty}_{0}((-a,T)\times \mathbb{R})$ we have 
\begin{align*}
\int_{0}^{T}\int_{\mathbb{R}}u_{n}(\tau,x)\partial_{\tau}\varphi(\tau,x)\ \dd x\dd \tau&+\int_{0}^{T}\int_{\mathbb{R}}\partial_{x}\varphi(\tau,x)V_{1}(V_{2}(u_{n})\ast  \eta)u_{n}(\tau,x)\ \dd x\dd \tau\\
&+\int_{\mathbb{R}}u^{0}_{n}(x)\varphi(0,x)\ \dd x=0.     
\end{align*}
That 
\begin{align*}
\int_{0}^{T}\int_{\mathbb{R}}u_{n}(\tau,x)\partial_{\tau}\varphi(\tau,x)\ \dd x\dd \tau\underset{n\rightarrow \infty}{\rightarrow}\int_{0}^{T}\int_{\mathbb{R}}u(\tau,x) \partial_{\tau}\varphi(\tau,x)\ \dd x\dd \tau  \end{align*}
and 
\begin{align*}
\int_{\mathbb{R}}u^{0}_{n}(x)\varphi(0,x)\ \dd x\underset{n\rightarrow \infty}{\rightarrow} \int_{\mathbb{R}}u^{0}(x)\varphi(0,x)\ \dd x    
\end{align*}
is clear. In addition, denoting $\Omega=\mathrm{supp}(\varphi)\cap [0,T]\times \mathbb{R}$, it holds that 
\begin{align*}
&\left\vert \int_{0}^{T}\int_{\mathbb{R}}V_{1}(V_{2}(u_{n})\ast \eta) u_{n}(\tau,x)\partial_{x}\varphi(\tau,x) \ \dd x\dd \tau-\int_{0}^{T}\int_{\mathbb{R}}V_{1}(V_{2}(u)\ast \eta)u(\tau,x)\partial_{x}\varphi(\tau,x)\ \dd x\dd \tau\right\vert\\
&\leq C\left\Vert V_{2}(u_{n})\ast \eta-V_{2}(u)\ast \eta\right\Vert_{L^{1}(\Omega)}+\underset{n\rightarrow \infty}{\underbrace{\left\vert \int_{0}^{T}\int_{\mathbb{R}}V_{1}(V_{2}(u)\ast \eta)(u_{n}-u)(\tau,x)\partial_{x}\varphi(\tau,x)\ \dd x\dd \tau\right\vert}}.   
\end{align*}
Using the previous step we have 
\begin{align}
\left\Vert V_{2}(u_{n})\ast \eta-V_{2}(u)\ast \eta\right\Vert_{L^{1}(\Omega)}\underset{n\rightarrow \infty}{\rightarrow} 0,    
\end{align}
which shows that $u$ is a weak solution. \\
\textbf{Uniqueness}. Let $u_{1},u_{2}$ be bounded weak solutions with the same initial data 
$u^{0}$. For brevity, for $i=1,2$ set $U_{i}\coloneqq V_{2}(u_{i})$ and $W_{i}\coloneqq U_{i}\ast \eta$ and consider the associated characteristics $X_{i}$ ($i=1,2$), i.e. the solutions of 
\begin{align*}
\frac{\dd}{\dd t}X_{i}(t,x)=V_{1}(W_{i}(X_{i}(t,x))),\ X_{i}(0,x)=x.    
\end{align*}
In addition, set $Y_{i}(t,x)\coloneqq \partial_{x}X_{i}(t,x)$.  The proof of uniqueness will be  based on obtaining a Gr\"onwall estimate for the time dependent quantity $Q(t)$ defined by 
\begin{align*}
Q(t)\coloneqq \left\Vert (X_{1}-X_{2})(t,\cdot)\right\Vert_{\infty}+\left\Vert (Y_{1}-Y_{2})(t,\cdot)\right\Vert_{\infty}+\left\Vert u_{1}(t,X_{1}(t,\cdot))-u_{2}(t,X_{2}(t,\cdot))\right\Vert_{\infty}.        
\end{align*}
As usual, in the sequel $C>0$ designates any "harmless" constant that may vary between the different equations.\\
\textbf{1}. In order to obtain an evolution estimate on $Q(t)$ we will need the following preliminary estimate. We claim that for a locally Lipschitz function $V$ (which will eventually be taken to be $V_{1}'$) it holds that:  
\begin{align}
&\left\Vert V(W_{1}(t,X_{1}(t,\cdot)))(W_{1}-U_{1})(t,X_{1}(t,\cdot))-V(W_{2}(t,X_{2}(t,\cdot)))(W_{2}-U_{2})(t,X_{2}(t,\cdot))\right\Vert_{\infty} \notag\\
&\leq C\left(\left\Vert (X_{1}-X_{2})(t,\cdot)\right\Vert_{\infty}+\left\Vert (Y_{1}-Y_{2})(t,\cdot)\right\Vert_{\infty}+\left\Vert u_{1}(t,X_{1}(t,\cdot))-u_{2}(t,X_{2}(t,\cdot))\right\Vert_{\infty}\right). 
\label{general estimate}
\end{align}
To prove \eqref{general estimate}, we first observe that  
\begin{align}
&\left\Vert V(W_{1}(t,X_{1}(t,\cdot)))(W_{1}-U_{1})(t,X_{1}(t,\cdot))-V(W_{2}(t,X_{2}(t,\cdot)))(W_{2}-U_{2})(t,X_{2}(t,\cdot))\right\Vert_{\infty} \notag\\ 
&\leq \left\Vert (V(W_{1}(t,X_{1}(t,\cdot)))-V(W_{2}(t,X_{2}(t,\cdot))))(W_{1}-U_{1})(t,X_{1}(t,\cdot))\right\Vert_{\infty} \notag\\
&+\left\Vert V(W_{2}(t,X_{2}(t,\cdot)))(W_{1}(t,X_{1}(t,\cdot))-W_{2}(t,X_{2}(t,\cdot))+U_{2}(t,X_{2}(t,\cdot))-U_{1}(t,X_{1}(t,\cdot)))\right\Vert_{\infty} \notag\\ 
&\leq C\left(\left\Vert W_{1}(t,X_{1}(t,\cdot))-W_{2}(t,X_{2}(t,\cdot))\right\Vert_{\infty}+\left\Vert U_{1}(t,X_{1}(t,\cdot))-U_{2}(t,X_{2}(t,\cdot))\right\Vert_{\infty}\right). \label{first est for pre} 
\end{align}
Clearly we have 
\begin{align}
\left\Vert U_{1}(t,X_{1}(t,\cdot))-U_{2}(t,X_{2}(t,\cdot))\right\Vert_{\infty}\leq C\left\Vert u_{1}(t,X_{1}(t,\cdot))-u_{2}(t,X_{2}(t,\cdot))\right\Vert_{\infty}.  \label{easy est}     
\end{align}
The estimate for the term $\left\Vert W_{1}(t,X_{1}(t,\cdot))-W_{2}(t,X_{2}(t,\cdot))\right\Vert_{\infty}$ is more delicate. Expanding the expression for $W_{i}$ we get  
\begin{align}
&W_{1}(t,X_{1}(t,x))-W_{2}(t,X_{2}(t,x)) \notag\\
&=\int_{X_{1}(t,x)}^{\infty}e^{X_{1}(t,x)-y}U_{1}(t,y)\ \dd y-\int_{X_{2}(t,x)}^{\infty}e^{X_{2}(t,x)-y}U_{2}(t,y)\ \dd y \notag\\
&=e^{X_{1}(t,x)}\int_{X_{1}(t,x)}^{\infty}e^{-y}U_{1}(t,y)\ \dd y-e^{X_{2}(t,x)}\int_{X_{2}(t,x)}^{\infty}e^{-y}U_{2}(t,y)\ \dd y \notag\\
&=(e^{X_{1}(t,x)}-e^{X_{2}(t,x)})\int_{X_{1}(t,x)}^{\infty}e^{-y}U_{1}(t,y)\ \dd y \notag\\
&+e^{X_{2}(t,x)}\left(\int_{X_{1}(t,x)}^{\infty}e^{-y}U_{1}(t,y)\ \dd y-\int_{X_{2}(t,x)}^{\infty}e^{-y}U_{2}(t,y)\ \dd y\right)\coloneqq I_{1}+I_{2}. 
\end{align}
To estimate $\left\vert I_{1} \right\vert$, note that 
\begin{align*}
\left\vert \int_{X_{1}(t,x)}^{\infty}e^{-y}U_{1}(t,y)\ \dd y\right\vert \leq C\int_{X_{1}(t,x)}^{\infty}e^{-y}\ \dd y=Ce^{-X_{1}(t,x)}.   
\end{align*}
Thus, writing $X_{i}$ as 
\begin{align*}
X_{i}(t,x)=x+\int_{0}^{t}V_{1}(W_{i}(\tau,X_{i}(\tau,x)))\ \dd \tau     
\end{align*}
we see that   
\begin{align}
\left\vert I_{1}\right\vert &\leq Ce^{-X_{1}(t,x)}\left\vert e^{X_{1}(t,x)}-e^{X_{2}(t,x)}\right\vert \notag\\
&= C\exp\left(-x-\int_{0}^{t}V_{1}(W_{1})(\tau,X_{1}(\tau,x))\ \dd \tau\right)\times \notag\\
&\left\vert \exp\left(x+\int_{0}^{t}V_{1}(W_{1})(\tau,X_{1}(\tau,x))\ \dd \tau\right)-\exp\left(x+\int_{0}^{t}V_{1}(W_{2})(\tau,X_{2}(\tau,x))\ \dd \tau\right)\right\vert \notag\\
&\leq  C\left\vert \exp\left(\int_{0}^{t}V_{1}(W_{1})(\tau,X_{1}(\tau,x))\ \dd \tau\right)-\exp\left(\int_{0}^{t}V_{1}(W_{2})(\tau,X_{2}(\tau,x))\ \dd \tau\right)\right\vert \notag\\
&\leq C\left\vert \int_{0}^{t}V_{1}(W_{1})(\tau,X_{1}(\tau,x))\ \dd \tau-\int_
{0}^{t}V_{1}(W_{2})(\tau,X_{2}(\tau,x))\ \dd \tau\right\vert \notag\\
&=C\left\vert (X_{1}-X_{2})(t,x)\right\vert\leq C\left\Vert (X_{1}-X_{2})(t,\cdot)\right\Vert_{\infty}.    \label{I1 estimate}  
\end{align}
We move to obtain a bound on $I_{2}$. By the change of variables $y=X_{1}(t,z)$ and $y=X_{2}(t,z)$ the term $\left\vert I_{2}\right\vert $ is recast as 
\begin{align}
\left\vert I_{2}\right\vert&=e^{X_{2}(t,x)}\left\vert\int_{x}^{\infty}e^{-X_{1}(t,z)}U_{1}(t,X_{1}(t,z))Y_{1}(t,z)\ \dd z-\int_{x}^{\infty}e^{-X_{2}(t,z)}U_{2}(t,X_{2}(t,z))Y_{2}(t,z)\ \dd z \right\vert \notag \\
&\leq e^{X_{2}(t,x)}\left\vert \int_{x}^{\infty}(e^{-X_{1}(t,z)}-e^{-X_{2}(t,z)})U_{1}(t,X_{1}(t,z))Y_{1}(t,z)\ \dd z\right\vert \notag\\
&+e^{X_{2}(t,x)}\left\vert \int_{x}^{\infty}e^{-X_{2}(t,z)}\left(U_{1}(t,X_{1}(t,z))Y_{1}(t,z)-U_{2}(t,X_{2}(t,z))Y_{2}(t,z)\right)\ \dd z\right\vert.  \label{I2 first est} 
\end{align}
To bound the first term in \eqref{I2 first est} we first observe that  
\begin{align*}
&e^{X_{2}(t,x)}(e^{-X_{1}(t,z)}-e^{-X_{2}(t,z)})\\
&=\exp\left(\int_{0}^{t}V_{1}(W_{2}(\tau,X_{2}(\tau,x)))\ \dd \tau\right)e^{x-z}\\
&\times \left[\exp\left(-\int_{0}^{t}V_{1}(W_{1}(\tau,X_{1}(\tau,z)))\ \dd \tau\right)-\exp\left(-\int_{0}^{t}V_{1}(W_{2}(\tau,X_{2}(\tau,z)))\ \dd \tau\right)\right].    
\end{align*}
Therefore the first term in \eqref{I2 first est} is bounded by 
\begin{align*}
&\left(C\int_{x}^{\infty}e^{x-z}\ \dd z\right)\times \left\Vert \exp\left(-\int_{0}^{t}V_{1}(W_{1}(\tau,X_{1}(\tau,\cdot)))\ \dd \tau\right)-\exp\left(-\int_{0}^{t}V_{1}(W_{2}(\tau,X_{2}(\tau,\cdot)))\ \dd \tau\right)\right\Vert_{\infty}\\
&\leq C\left\Vert X_{1}(t,\cdot)-X_{2}(t,\cdot)\right\Vert_{\infty}.  
\end{align*}
As for the second term in \eqref{I2 first est}, note that $e^{X_{2}(t,x)-X_{2}(t,z)}\leq Ce^{x-z}$ 
so that it is bounded by 
\begin{align*}
&C\int_{x}^{\infty}e^{x-z}\left\vert (U_{1}(t,X_{1}(t,z))Y_{1}(t,z)-U_{2}(t,X_{2}(t,z))Y_{2}(t,z)\right\vert\ \dd z\\
&\leq C\int_{x}^{\infty}e^{x-z}\left\vert U_{1}(t,X_{1}(t,z))-U_{2}(t,X_{2}(t,z))\right\vert\left\vert Y_{1}\right\vert(t,z)\dd z\\
&+C\int_{x}^{\infty}e^{x-z}\left\vert U_{2}\right\vert(t,X_{2}(t,z))\left\vert Y_{1}-Y_{2}\right\vert(t,z)\ \dd z
\\ &\leq C\left\Vert u_{1}(t,X_{1}(t,\cdot))-u_{2}(t,X_{2}(t,\cdot))\right\Vert_{\infty} +C\left\Vert (Y_{1}-Y_{2})(t,\cdot)\right\Vert_{\infty}.  
\end{align*}
We therefore conclude that 
\begin{align}
\left\vert I_{2} \right\vert \leq C\left(\left\Vert (X_{1}-X_{2})(t,\cdot)\right\Vert_{\infty}+\left\Vert (Y_{1}-Y_{2})(t,\cdot)\right\Vert_{\infty}+\left\Vert u_{1}(t,X_{1}(t,\cdot))-u_{2}(t,X_{2}(t,\cdot))\right\Vert_{\infty}\right)=CQ(t). \label{I2boun}    
\end{align}
Combining \eqref{I1 estimate} and \eqref{I2boun} we get 
\begin{align}
\left\Vert W_{1}(t,X
_{1}(t,\cdot))-W_{2}(t,X_{2}(t,\cdot))\right\Vert_{\infty}\leq CQ(t).\label{hardest}      
\end{align}
Substituting \eqref{easy est} and \eqref{hardest} inside \eqref{first est for pre} yields \eqref{general estimate}.\\ 
\textbf{2}. In this step we estimate $\left\Vert u_{1}(t,X_{1}(t,\cdot))-u_{2}(t,X_{2}(t,\cdot))\right\Vert_{\infty}$. Start by writing  
\begin{align}
u_{i}(t,X_{i}(t,x))=u^{0}(x)\exp\left(-\int_{0}^{t}\partial_{2}(V_{1}(W_{i}))(\tau,X_{i}(\tau,x))\ \dd \tau\right). \label{representation formula for solutions}    
\end{align}
In view of \eqref{representation formula for solutions} we have 
\begin{align*}
&\left\Vert u_{1}(t,X_{1}(t,\cdot))-u_{2}(t,X_{2}(t,\cdot))\right\Vert_{\infty} \\
&\leq \left\Vert u^{0} \right\Vert_{\infty}\left\Vert \exp\left(-\int_{0}^{t}\partial_{2}(V_{1}(W_{1}))(\tau,X_{1}(\tau,\cdot))\ \dd\tau\right)-\exp\left(-\int_{0}^{t}\partial_{2}(V_{1}(W_{2}))(\tau,X_{2}(\tau,\cdot))\ \dd \tau\right)\right\Vert_{\infty}\\
&\leq C\int_{0}^{t}\left\Vert \partial_{2}(V_{1}(W_{1}))(\tau,X_{1}(\tau,\cdot))-\partial_{2}(V_{1}(W_{2}))(\tau,X_{2}(\tau,\cdot))\right\Vert_{\infty}\ \dd \tau\\
&= C\int_{0}^{t}\left\Vert V_{1}'(W_{1}(\tau,X_{1}(\tau,\cdot)))(W_{1}-U_{1})(\tau,X_{1}(\tau,\cdot))-V_{1}'(W_{2}(\tau,X_{2}(\tau,\cdot)))(W_{2}-U_{2})(\tau,X_{2}(\tau,\cdot))\right\Vert_{\infty}\ \dd \tau. 
\end{align*}
Thanks to \eqref{general estimate} (applied with $V=V_{1}'$) we conclude that 
\begin{align}
\left\Vert u_{1}(t,X_{1}(t,\cdot))-u_{2}(t,X_{2}(t,\cdot))\right\Vert_{\infty}\leq C\int_{0}^{t}Q(\tau)\ \dd \tau. \label{u1-u2 est}     
\end{align}
\textbf{3}. In this step we estimate $\left\Vert (X_{1}-X_{2})(t,\cdot)\right\Vert_{\infty}$. We have  
\begin{align*}
\left\Vert (X_{1}-X_{2})(t,\cdot)\right\Vert_{\infty} &\leq \int_{0}^{t}\left\Vert V_{1}(W_{1}(\tau,X_{1}(\tau,\cdot)))-V_{1}(W_{2}(\tau,X_{2}(\tau,\cdot)))\right\Vert_{\infty} \ \dd \tau\\
&\leq C\int_{0}^{t}\left\Vert W_{1}(\tau,X_{1}(\tau,\cdot))-W_{2}(\tau,X_{2}(\tau,\cdot))\right\Vert_{\infty}\ \dd \tau.  
\end{align*}
Thanks to \eqref{hardest} we obtain 
\begin{align}
\left\Vert (X_{1}-X_{2})(t,\cdot)\right\Vert_{\infty}\le C\int_{0}^{t}Q(\tau)\ \dd \tau. \label{X1-X2 est}      
\end{align}
\textbf{4.} In this step we estimate $\left\Vert (Y_{1}-Y_{2})(t,\cdot)\right\Vert_{\infty}$. We have 
\begin{align}
(Y_{1}-Y_{2})(t,x)=&\int_{0}^{t}V_{1}'(W_{1}(\tau,X_{1}(\tau,x)))(W_{1}(\tau,X_{1}(\tau,x))
-U_{1}(\tau,X_{1}(\tau,x)))Y_{1}(\tau,x)\ \dd\tau \notag\\
&-\int_{0}^{t}V_{1}'(W_{2}(\tau,X_{2}(\tau,x)))(W_{2}(\tau,X_{2}(\tau,x))-U_{2}(\tau,X_{2}(\tau,x)))Y_{2}(\tau,x)\ \dd \tau \notag\\ 
=&\int_{0}^{t}V_{1}'(W_{1}(\tau,X_{1}(\tau,x)))\left(W_{1}(\tau,X_{1}(\tau,x))-U_{1}(\tau,X_{1}(\tau,x))\right)(Y_{1}-Y_{2})(\tau,x)\ \dd \tau \notag\\
&+\int_{0}^{t}Y_{2}(\tau,x)\Big(V_{1}'(W_{1}(\tau,X_{1}(\tau,x)))(W_{1}(\tau,X_{1}(\tau,x))-U_{1}(\tau,X_{1}(\tau,x))) \notag\\
&-V_{1}'(W_{2}(\tau,X_{2}(\tau,x)))(W_{2}(\tau,X_{2}(\tau,x))-U_{2}(\tau,X_{2}(\tau,x)))\Big)\ \dd \tau. \label{diff of Y1 Y2}
\end{align}
The first integral in the right-hand side of \eqref{diff of Y1 Y2} is bounded by 
\begin{align*}
C\int_{0}^{t}\left\Vert (Y_{1}-Y_{2})(\tau,\cdot)\right\Vert_{\infty}\ \dd \tau.     
\end{align*}
By means of \eqref{general estimate}, the second integral in the right-hand side of \eqref{diff of Y1 Y2} is bounded by 
\begin{align*}
C\int_{0}^{t}Q(\tau)\ \dd \tau.    
\end{align*}
So to conclude, we obtain 
\begin{align}
\left\Vert (Y_{1}-Y_{2})(t,\cdot)\right\Vert_{\infty}\leq C\int_{0}^{t}Q(\tau)\ \dd \tau. \label{Y1-Y2 est}      
\end{align}
\textbf{5}. Gathering \eqref{u1-u2 est}, \eqref{X1-X2 est} and \eqref{Y1-Y2 est} we deduce 
\begin{align*}
Q(t)\leq C\int_{0}^{t}Q(\tau)\ \dd \tau.     
\end{align*}
By Gr\"onwall's inequality it follows that $Q(t)\equiv 0$, which in turn implies $X_{1}(t,x)\equiv X_{2}(t,x)$ and $u_{1}(t,X_{1}(t,x))\equiv u_{2}(t,X_{2}(t,x))$ and hence $u_{1}\equiv u_{2}$. 
\qed
\begin{rem}
The proof of Theorem \ref{approximation lemma} is essentially a repetition of the existence proof in the argument outlined above.    
\end{rem}

\end{document}